%% file: main.tex
\documentclass{article}
\newif\ifmarkup
\markupfalse

\newcommand{\markupadd}[1]{%
\ifmarkup
\textcolor{blue}{#1}%
\else
#1%
\fi
}

\newcommand{\markupdelete}[1]{%
\ifmarkup
\textcolor{red}{#1}%
\else
\fi
}

\input{prologue.tex}

\title{Optimization on a Finer Scale:\\ Bounded Local Subgradient Variation Perspective}
\author{Jelena Diakonikolas\footnote{
Department of Computer Sciences, 
University of Wisconsin-Madison, 
\texttt{jelena@cs.wisc.edu}}
\and Crist\'{o}bal Guzm\'{a}n\footnote{
Institute for Mathematical and Computational Engineering, 
              Faculty of Mathematics and School of Engineering, 
              Pontificia Universidad Cat\'olica de Chile, 
              \texttt{crguzmanp@mat.uc.cl}}
              }
\date{}

\begin{document}
\maketitle

\begin{abstract}
We initiate the study of nonsmooth optimization problems under bounded local subgradient variation, which postulates bounded difference between (sub)gradients in small local regions around points, in either average or maximum sense. The resulting class of objective functions encapsulates the classes of objective functions traditionally studied in the optimization literature, which are defined based on either Lipschitz continuity of the objective or H\"{o}lder/Lipschitz continuity of the function's gradient. Further, the defined class is richer in the sense that it contains functions that are neither Lipschitz continuous nor have a H\"{o}lder continuous gradient. Finally, when restricted to the aforementioned traditional classes of optimization problems, the constants defining the studied classes lead to more fine-grained oracle complexity bounds. 
Some highlights of our results are that: (i) it is possible to obtain complexity results for both convex and nonconvex optimization problems with (local or global) Lipschitz constant being replaced by a constant of local subgradient variation, corresponding to small local regions and (ii) complexity of the subgradient set around the set of optima -- measured by its mean width in a local region around optima -- plays a role in the complexity of nonsmooth optimization, particularly in parallel optimization settings. 
A consequence of (ii) is that for any error parameter $\epsilon > 0$, parallel oracle complexity of nonsmooth Lipschitz convex optimization is lower than its sequential oracle complexity by a factor $\tilde{\Omega}\big(\frac{1}{\epsilon}\big)$ whenever the objective function is piecewise-linear with the number of pieces  polynomial in the dimension and $1/\epsilon.$ This is particularly surprising considering that existing parallel complexity lower bounds are based on such classes of functions. The seeming contradiction is resolved by considering the region in which  the algorithm is allowed to query the objective. %
\end{abstract}

\section{Introduction}

Nonsmooth optimization problems pose some of the most intricate challenges within the realm of continuous optimization. 
As a result, they have been intensively studied from the algorithm design %
perspective since at least the 1960s \cite{bagirov2014introduction,hosseini2019nonsmooth,lemarechal1978nonsmooth}. The study of nonsmooth optimization concerns solving minimization problems of the form
\begin{equation}\label{eq:prob}\tag{P}
    \min_{\vx \in \cx}f(\vx),
\end{equation}
where $f$ is a (typically Lipschitz) continuous function (or satisfies other structural properties) and is not necessarily everywhere differentiable. Within this work, we are concerned with functions that are \emph{locally} Lipschitz-continuous, in the sense that they have bounded Lipschitz constants on compact sets, but, importantly, \markupadd{we do not impose global upper bounds on those Lipschitz constants.} \markupdelete{we make no assumptions about the values of those Lipschitz constants.} We further focus on standard settings where $\cx$ is closed, convex, nonempty, and admits efficiently computable projections.  

It was noted very early on that even though the objectives in such nonsmooth optimization problems are continuous (and, as such, differentiable almost everywhere, as a consequence of the classical Rademacher's theorem \cite{Rademacher:1919}), traditional methods developed for smooth optimization generally fail to converge when applied to them as a black box. It was formally established in the subsequent literature that in terms of oracle-based worst-case complexity, (Lipschitz-continuous) nonsmooth optimization problems are more challenging  than smooth optimization problems, both in the settings of convex~\cite{nemirovski:1983} and nonconvex  \cite{kornowski2022oracle} optimization, unless additional assumptions about the structure \cite{carmon2021thinking,nesterov2005smooth} and/or oracle access to $f$ \cite{Guler:1991,Guler:1992,carmon2021thinking,moreau1965proximite,martinet1970regularisation,rockafellar1976monotone} %
are made and crucially used in the algorithm design and analysis.  

Despite the computational barriers preventing  algorithmic speedups of nonsmooth optimization in the worst case, common nonsmooth optimization problems are often shown to be solvable with faster converging algorithms (even without access to stronger oracles such as the proximal oracle), sometimes even exhibiting linear (i.e., with geometrically reducing error) convergence \emph{locally} and/or \emph{in practice} \cite{davis2022nearly,han2023survey}. This large gap between the worst-case lower bounds and empirical performance on common instances prompts the question: 
\begin{center}
 \em   What type of structure makes certain nonsmooth optimization problems easier than others, and what kind of algorithms effectively exploit such structural properties?%
\end{center}

 \subsection{Motivation \& Intuition}

 \begin{figure}[t]
    \centering
    \subfigure[``farther from smooth'']{\includegraphics[width=.3\textwidth]{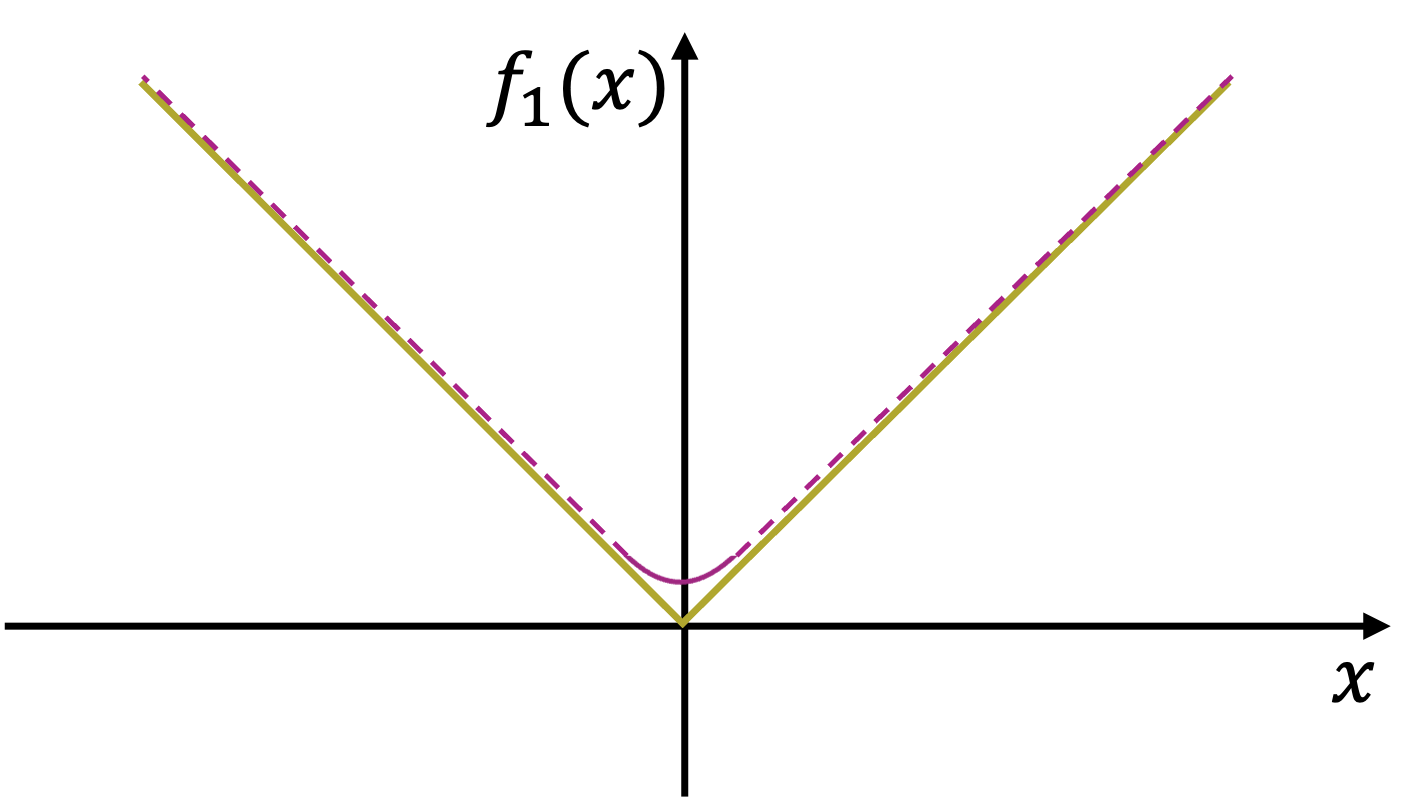}}\hspace{.55in}
    \subfigure[``closer to smooth'']{\includegraphics[width=.3\textwidth]{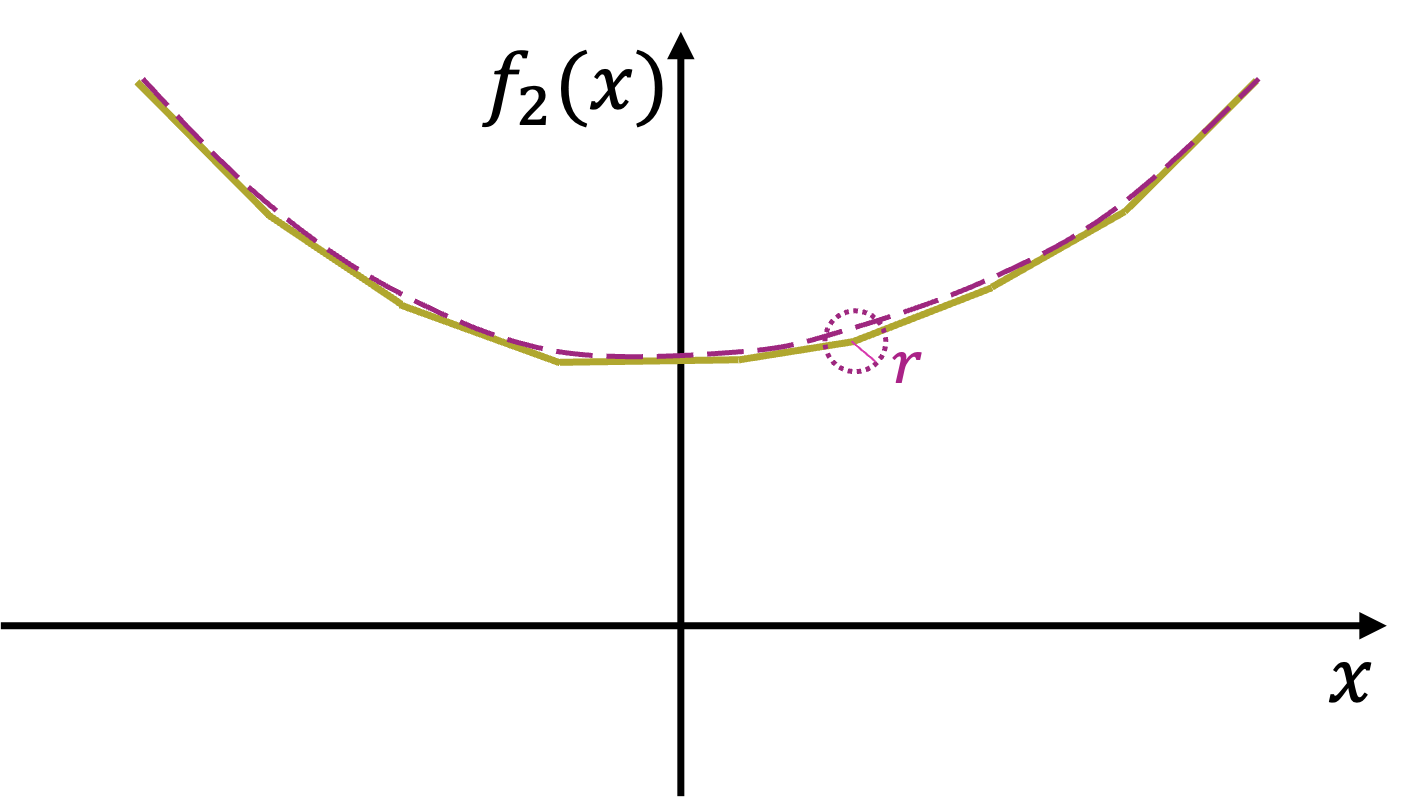}}
    \caption{Two nonsmooth, piecewise-linear, functions with the same Lipschitz constant. If both functions are averaged over a radius $r$ shown on the right plot, the right function remains much closer to its averaged, smoothed, variant indicated by the dashed line. Further, on any local neighborhood of radius $r$, the right function exhibits much milder changes in the slope than the left function does on its (only) kink.}
    \label{fig:GradBMO-comparison}
\end{figure}

Our starting point for the research presented in this work came from observations illustrated in Figure~\ref{fig:GradBMO-comparison}, which depicts two nonsmooth, Lipschitz-continuous,  piecewise-linear functions with the same Lipschitz constant. Visually, it is apparent that the right function is closer to a smooth function than the left function, in the sense that averaging both functions over small intervals around each point to obtain their smoothed counterparts, the right function remains much closer to its smoothing than the left one, in terms of the maximum deviation. Thus, intuitively, the right function should be easier to optimize, as the oracle complexity of smooth optimization is lower than the oracle complexity of nonsmooth optimization. However, this is not captured by  existing results.\footnote{Note, for example, that the most common interpolation between smooth and nonsmooth convex functions, given by H\"older continuity of gradients with exponent $\kappa\in(0,1]$ \cite{Nemirovski:1985,Guzman:2015}, is in both functions from Figure \ref{fig:GradBMO-comparison} infinite, for any exponent $\kappa>0$. Hence, this interpolation does not provide a satisfactory quantification of the complexity.} 

Taking a closer look at the examples in Figure~\ref{fig:GradBMO-comparison}, we can notice that the property that differentiates these two functions and makes one closer to its smooth approximation than the other is how much the slope of the function varies across small regions. This basic observation is the main motivation for introducing the notions of local subgradient variation -- corresponding to variation on \emph{average} and in  the worst-case, \emph{maximum} sense -- and studying oracle complexity of nonsmooth optimization under these notions of subgradient variation.  

Our work is importantly motivated by the limitations encountered in the extensively studied technique of randomized smoothing (see e.g., \cite{duchi2012randomized,nesterov2017random}), that uses convolution with a (uniform over a ball or Gaussian) kernel to reduce a nonsmooth (deterministic) optimization problem to a stochastic smooth optimization one, or even a problem only requiring access to a stochastic zeroth-order oracle. It is  well-known  that in this setting \markupadd{the} oracle complexity necessarily scales polynomially with the dimension~\cite{duchi2012randomized,kornowski2022oracle,diakonikolas2020lower,nemirovski1994parallel}, 
which severely limits its applicability in high-dimensional problems. As another notion of fine-grained complexity, we demonstrate in this work that it is the complexity of the subgradient around a function minimizer $\vx_*$ (assumed to exist in this work) that determines whether this dependence on the dimension can be improved. In particular, for a ``simple'' subgradient set around $\vx_*$, the dependence on the dimension $d$ can be brought down to $(\ln(d))^{1/4}$ or even a constant. This is further discussed in the next subsection and in Section~\ref{sec:randomized-smoothing}. 

Before moving onto describing our main results, we highlight the following properties of the functions that have bounded local variation of the subgradient (a more precise discussion is provided in Section~\ref{sec:bounded-local-var}, where these notions are formally introduced). First, every Lipschitz function has bounded local variation of the subgradient (under any of the considered average/maximum criteria), with the constant of subgradient variation being larger than the Lipschitz constant by at most a factor of 2, but possibly being arbitrarily smaller. In particular, the converse to this statement is not true: there are functions that have bounded local subgradient variation but are not Lipschitz-continuous (see Example \ref{ex:nonsmooth-quadratic-growth} in Section~\ref{sec:examples}). Thus, the class of functions with bounded local subgradient variation strictly contains the class of Lipschitz-continuous functions. Second, bounded local subgradient variation does not preclude superlinear (including quadratic) growth of a function, which clearly is not true for functions that are only Lipschitz continuous. %
Finally, we note that while the notion of subgradient variation is not new and it has been explored at least in the \emph{maximum} and \emph{global} sense \cite{Nesterov:2005minimizing,nesterov2015universal,Devolder:2014}, the key insight of our work is that it suffices for such a property to hold only in a \emph{local} sense, between points sufficiently close to each other.  

\subsection{Main Results}

 Our main findings are summarized as follows. 

\paragraph{Local subgradient variation as a measure of complexity} 
The main contribution of this work is 
initiating the study of new classes of objectives in the context of (convex and nonconvex) nonsmooth and (weakly) smooth optimization.  These classes of objectives are called {\em bounded maximum local variation of subgradients} (or \gradBMV, for short) and {\em bounded mean local oscillation of subgradients} (or  \gradBMO, for short), and are introduced in Section  \ref{sec:bounded-local-var}. \gradBMO\ is the weaker of these two properties, in the sense that it is implied by \gradBMV, while the converse does not hold in general. We provide different characterizations of structural properties of functions under these two notions of local subgradient variation and systematically investigate the complexity of classes of \gradBMV\ and \gradBMO\ problems, with the focus on demonstrating that these weak regularity properties suffice for obtaining a more fine-grained characterization of oracle complexity. Notably, because \gradBMV\ and \gradBMO\ functions are not necessarily (globally) Lipschitz continuous (but all Lipschitz-continuous functions are both \gradBMV\ and \gradBMO\ with a constant at most 2 times larger and possibly much smaller), on a conceptual level, our results demonstrate that weaker properties than Lipschitz continuity suffice for tractability of optimization. 

\paragraph{Deterministic convex optimization under \gradBMV} We begin our discussion of oracle upper bounds 
by treating \gradBMV\ functions as approximately smooth functions. While this idea is not new and was used in \cite{Devolder:2014,nesterov2015universal} as a means of handling weakly smooth functions in a universal manner, the prior work has only considered the settings in which such properties hold in a \emph{global} sense, between any pair of points. By contrast, we demonstrate that only a \emph{local} such property, assumed to hold only between points at distance at most $r > 0$, suffices. As a result, we obtain results that are similar to those in \cite{Devolder:2014,nesterov2015universal}, but with a parameter %
that is potentially much lower than the Lipschitz/weak smoothness constants from prior work, 
thus providing a more fine-grained characterization of oracle complexity. 

\paragraph{Randomized and possibly parallel convex optimization under \gradBMV\ and \gradBMO} As mentioned before, one of our main initial goals in this work was to understand and possibly remove the computational barriers of randomized smoothing, which introduces polynomial dependence on the dimension in the oracle complexity bound. If one further considers parallel settings, where $\mathrm{poly}(d, 1/\epsilon)$ queries may be asked in parallel per round of computation, then the term with the polynomial dependence on the dimension dominates the complexity (measured as the number of sequential rounds of queries), and this polynomial dependence is unavoidable in the worst  case~\cite{diakonikolas2020lower,balkanski2019exponential,bubeck2019complexity,woodworth2018graph,nemirovski1994parallel}. As another measure of fine-grained complexity, we show that this worst-case polynomial dependence on the dimension is determined by the complexity of the subdifferential set around a minimizer $\vx_* \in \argmin_\vx f(\vx).$ In particular, let $\partial f_r(\vx_*)$ be the convex hull of the subgradients in the Euclidean ball of radius $r = O(\epsilon \sqrt{d})$, centered at $\vx_*.$ The diameter of this set is determined by the \gradBMV\ constant, denoted by $\Llocr.$ Let $K_{\vx_*}$ be a polytope of diameter $O(\Llocr)$ that contains $\partial f_r(\vx_*)$ and has the smallest number of vertices. Then, the dimension-dependent term in the oracle complexity can be bounded by $O([\ln(\ext(K_{\vx_*}))]^{1/4}),$ where $\ext(K_{\vx_*})$ denotes the number of vertices of $K_{\vx_*}.$ As a result, when $\ext(K_{\vx_*}) = \mathrm{poly}(d),$\footnote{\markupadd{An example where the number of vertices of the subdifferential is moderate is the structured convex program $\min\{f(\vx)=\max_{\vy\in \Delta_{K}} \vy^{\top}(\mA\vx-\vb): \vx \in {\cal X}\}$, studied in the seminal works \cite{Nesterov:2005minimizing,nemirovski2004prox}. Here, ${\cal X}\subseteq \rr^d$, and $\Delta_K$ is the standard simplex in dimension $K$, where $K$ is scales polynomially with $d$. While our upper bounds would not match the results of these works, their algorithms heavily leverage the structure of the problem, whereas our results apply using only a black-box subgradient oracle access, with no explicit knowledge of the problem structure.}} the dependence on the dimension can be brought down to at most $O((\ln(d))^{1/4}),$ which is nearly dimension-independent.\footnote{For example, when $d = 10^{12},$ $(\ln(d))^{1/4}\approx 2$.} Note that this result does not contradict prior lower bounds for parallel convex optimization \cite{diakonikolas2020lower,balkanski2019exponential,bubeck2019complexity,woodworth2018graph}, as is explained in Section~\ref{sec:randomized-smoothing}.%

\paragraph{Nonconvex optimization under \gradBMV} Our final set of results concerns nonconvex nonsmooth optimization. We build on recent results for Goldstein's method \cite{davis2022gradient,zhang2020complexity} and show that the local Lipschitz constant used in the past work can be replaced by the generally much smaller local constant of \gradBMV. In obtaining this result, we generalize and further simplify the analysis from \cite{davis2022gradient}.  %

\subsection{Related Work}

The complexity of continuous optimization is an actively investigated problem since the 1970s \cite{nemirovski:1983}. One of the main achievements of this theory is the precise quantification of the minimax optimal rates of convergence for smooth and nonsmooth convex optimization. However, as  argued earlier, this coarse parameterization based on the maximum Lipschitz constant (either for the objective or its gradient) misses much of the information that determines the difficulty of performing optimization. The goal of our work is to bridge this gap. Below we review various threads of research that are related to our work. %

\paragraph{Smoothing approaches} Both randomized and deterministic smoothing approaches have been widely used in nonlinear optimization for a long time and in different contexts \cite{stekloff1907expressions,beck2012smoothing,norkin2020stochastic,norkin2023constrained,yousefian2010convex,lakshmanan2008decentralized,duchi2012randomized,bubeck2019complexity,flaxman2005online,nemirovski:1983,nesterov2005smooth,nesterov2017random}. The idea of smoothing is a natural one: approximate a nonsmooth function by a smooth one and then apply methods for smooth minimization to the smoothed function. Some of the most basic examples are the use of Moreau envelope for deterministic methods (see, e.g., \cite{moreau1965proximite,rockafellar1976monotone}), which requires access to a proximal point oracle, and local randomized smoothing using a Gaussian kernel or uniform distribution on a ball or a sphere \cite{stekloff1907expressions,nemirovski:1983,flaxman2005online,nesterov2017random,duchi2012randomized} (all these kernels lead to very similar results, due to concentration properties of these distributions in high dimensions; see, e.g., \cite[Chapter 2]{blum2020foundations}).  
We note that the goal of our work is not to devise new smoothing approaches (in fact, we rely on a very simple randomized smoothing over a Euclidean ball), but to demonstrate usefulness of the introduced \gradBMO\ concept in proving oracle complexity upper bounds.  %

\paragraph{Local Lipschitzness and Relaxations} Much of the recent literature on nonsmooth optimization (e.g., \cite{davis2022nearly,davis2022gradient,hazan2015beyond}) replaces the global Lipschitz condition by a local one, which posits Lipschitz continuity on compact sets and implies bounded subgradients on those sets. This is based on the insight that many optimization algorithms ensure that their iterates remain on a compact set around optima, thus how a function behaves outside this set is irrelevant for optimization. An alternative definition is that of  Lipschitz continuity in local neighborhoods of points  \cite[Chapter 1]{cui2021modern}; however, this definition seems to have been primarily used to define and study generalized notions of derivatives rather than oracle complexity of optimization. 

We further highlight the following works which addressed nonsmooth problems that are not necessarily Lipschitz-continuous. In \cite{grimmer2019convergence},  a variant of projected subgradient method with normalized subgradients was analyzed, motivated by insights from \cite{shor2012minimization}. This work shows that the rate of convergence $1/\sqrt{k}$ can be established assuming that there is a nondecreasing nonnegative mapping $\cd: \rr_+ \to \rr_+ \cup \{+\infty\}$ such that $f(\vx) - f(\vx_*) \leq \cd(\|\vx- \vx_*\|_2)$ where $\vx_*$ is a fixed minimizer of $f$ and the complexity results are expressed in terms of this mapping. Such an assumption removes the requirement for regularity such as Lipschitz continuity to hold on a compact set between \emph{any} pair of points, but still requires at least some bounded growth condition to hold on the entire feasible set or on a sufficiently large ball around $\vx_*$. Another line of work \cite{renegar2016efficient,grimmer2018radial} develops a generic transformation from nonsmooth non-Lipschitz convex problems to convex Lipschitz problems and an algorithmic framework to address them at the cost of a simple line search (but removing projections in constrained settings) and with an error guarantee of the form $\frac{f(\vx) - f(\vx_*)}{f(\vx_0) + c - f(\vx_*)} \leq \epsilon,$ where $c > 0$ is a parameter of the algorithm. This corresponds to a multiplicative error guarantee for the shifted function $f(\cdot) - f(\vx_0) - c,$ but on the original problem the error is additive and equal to $(f(\vx_0) - f(\vx_*) + c)\epsilon.$ The resulting complexity bound replaces the usual dependence on the Lipschitz constant in traditional oracle complexity bounds by the inverse of a parameter defined by $R = \sup\{r \in \rr: f(\vx) \leq f(\vx_0) + c \text{ for all } \vx \text{ with } \|\vx - \vx_0\|_2 \leq r\}$. Given a fixed $c,$ for Lipschitz functions this constant is bounded by $c/M$, where $M$ is the Lipschitz constant of $f$ on the sublevel set $\{\vx: f(\vx) \leq f(\vx_0) + c\}.$ More generally, for the result to be informative, one needs to bound the growth of $f$ on a sufficiently large neighborhood of $\vx_0$ intersected with the sublevel set $\{\vx: f(\vx) \leq f(\vx_0) + c\}.$ In summary, all the results using notions of local Lipschitzness (and related concepts) that we are aware of require bounding the \emph{growth} of the function in a possibly small region. By contrast, our results rely on bounding the \emph{subgradient variation} in local regions. 

\paragraph{Stronger oracles} Improved complexity results for nonsmooth optimization are possible if access to additional oracles or structure of the problem is accessible to the algorithm. For example, there is vast literature on methods utilizing the proximal point oracle (see, e.g., \cite{rockafellar1976monotone,Guler:1991,Guler:1992,chambolle2011first}), requiring oracle access to minima of problems of the form $f(\vx) + \frac{1}{2\tau}\|\vx - \vx_0\|_2^2$ for any $\vx_0 \in \rr^d$ and $\tau >0.$ Another example in the recent literature is the ball optimization oracle \cite{carmon2020acceleration}, which gives the algorithm access to solutions of the problem $\min_{\vx: \|\vx - \vx_0\|_2 \leq r}f(\vx)$ for any $\vx_0 \in \rr^d$ and some $r > 0.$ Further, there are multiple results assuming that the objective function can be expressed as a structured maximization problem (e.g., arising from the convex conjugate of a function $f$ composed with a linear map $\mA \vx$ \cite{nesterov2005smooth,nemirovski2004prox,chambolle2011first} or the objective simply being a maximum of $N$ smooth or nonsmooth functions \cite{carmon2021thinking}) and where one is given oracle access to components of the said maximization problem (such as proximal point oracle access to the convex conjugate of $F$ in the case of $f(\vx) = F(\mA \vx)$ or first-order oracle access to component functions in the case where $f(\vx) = \max_{1 \leq i \leq n}f_i(\vx)$). Finally, it is possible to relax Lipschitz continuity by ``relative continuity'' -- where $\|\nabla f(\vx)\|_* \leq \frac{\sqrt{2M D_h(\vy, \vx)}}{\|\vy - \vx\|}$ for a pair of dual norms $\|\cdot\|, \|\cdot\|_*,$ a positive constant $M$, and Bregman divergence $D_h$ w.r.t.\ a reference function $h$ -- and recover the complexity results of nonsmooth Lipschitz convex optimization \cite{lu2019relative}. However, this approach requires oracle access to minimizers of $D_h(\vx, \vx_0) + \innp{\vz, \vx}$ for arbitrary but fixed $\vx_0, \vz$ (mirror descent steps), and, moreover, there are few examples of functions that satisfy such a relative continuity property. Our work does not require any  specialized oracles but only relies on the standard first-order oracle. %

\paragraph{Optimization algorithms} Finally, our work leverages existing algorithms and techniques in convex optimization \cite{cohen2018acceleration,gasnikov2018universal,nesterov2015universal} -- where we utilize universal accelerated algorithms, and nonconvex optimization \cite{zhang2020complexity,davis2022gradient} -- where we rely on a randomized version of Goldstein's method. 
For the latter, we note that 
implementations of first-order oracles based on the Goldstein subdifferential seemingly require randomization 
\cite{zhang2020complexity,kornowski2022oracle}, and some lower bounds support this idea \cite{jordan2023deterministic}. Recently, \cite{Kong:2023} has shown that for more particular   subclasses of nonconvex nonsmooth optimization problems, randomization can be avoided entirely. We note here that our focus is \emph{not} on algorithm design but on novel characterizations of oracle complexity in optimization. In particular, we demonstrate that existing algorithms can be analyzed for the complexity classes based on bounded local subgradient variation that we introduce and oracle complexity results that are both more general and more fine-grained can be obtained.

\section{Preliminaries}

Our primary focus is on the Euclidean space $(\rr^d, \|\cdot\|_2)$; however, much of the discussion extends to other normed spaces. We use $\cb := \{\vx \in \rr^d: \|\vx\|_2 \leq 1\}$ to denote the centered unit Euclidean ball and $\cb_r(\bar\vx) := \{\vx \in \rr^d: \|\vx-\bar\vx\|_2 \leq r\}$ to denote the Euclidean ball of radius $r$ centered at $\bar\vx$. When $\bar\vx=0$, we use the notation $\cb_r:=\cb_r(0)$. We use $\cs := \{\vx \in \rr^d: \|\vx\|_2 = 1\}$ to denote the centered unit Euclidean sphere. 

We say that a function $f: \rr^d \to \rr$ is $M$-Lipschitz continuous for some constant $M < \infty,$ if for any $\vx, \vy \in \rr^d,$
\begin{equation}\label{eq:M-Lip}
    |f(\vx) - f(\vy)| \leq M \|\vx - \vy\|_2. 
\end{equation}
If $f$ is additionally convex, then it is subdifferentiable on its whole domain; in particular $\partial f(\vx)\neq \emptyset$, for all $\vx\in\rr^d$. In such a case, we will denote for convenience by $\grad(\vx)$ an arbitrary (measurable w.r.t.~$\vx$) selection from $\partial f(\vx)$.  A similar conclusion holds without convexity (only under Lipschitzness), with the observation that here $f$ would be differentiable almost everywhere, thus a measurable selection would exist as well.

We say that a function $f:\rr^d\to\rr$ is locally Lipschitz if for every $r>0$, $f$ is Lipschitz over ${\cal B}_r$. Crucially, we make no assumptions about the Lipschitz constant on these balls: this is important as we would like to handle the case where the objective may not be globally Lipschitz with a uniform constant (e.g., a quadratic function). The main property of local Lipschitzness we need for our arguments is the almost sure differentiability of these functions, and the fundamental theorem of calculus. Both of them are stated below, for completeness.

\begin{theorem}[Rademacher \cite{Rademacher:1919}]
If $f:\rr^d\to\rr$ is locally Lipschitz then it is differentiable almost everywhere.
\end{theorem}

\begin{theorem}[Fundamental Theorem of Calculus (FTC)]\label{thm:FTC}
If $f:\rr^d\to \rr$ is locally Lipschitz, then for all $\va,\vb\in\rr^d$
\[ f(\vb)-f(\va)=\int_0^1 (\grad)_{\va,\vb}((1-t)\va+t\vb) \dd t, \]
where $(\grad)_{\va,\vb}$ is a measurable selection of the directional derivative $f^{\prime}(\cdot;\vb-\va)$. Moreover, if either $\va$ or $\vb$ are chosen generically,\footnote{\markupadd{Here and onwards, generically means for ``all points except for a Lebesgue-negligible set.''}}
\[ f(\vb)-f(\va)=\int_0^1 \langle \grad((1-t)\va+t\vb),\vb-\va\rangle \dd t. \] 
\end{theorem}

Finally, given $\delta>0$, we recall the definition of the Goldstein $\delta$-subdifferential of a (locally Lipschitz) function $f:\mathbb{R}^d\to\mathbb{R}$ at a point $\vx\in \mathbb{R}^d$, $\partial_{\delta}f(\vx)=\mbox{conv}\big(\{ \nabla f(\vy): \vy \in {\cal B}(\vx,r)\mbox{ and $f$ is differentiable at } \vy\}\big)$. We say that a point $\vx\in\mathbb{R}^d$ is $(\delta,\epsilon)$-stationary if $\mbox{dist}(\zeros, \partial_{\delta}f(\vx))\leq \epsilon$.

\section{Bounded Local Variation of the Subgradient}\label{sec:bounded-local-var}

In this section, we provide local regularity assumptions for the subgradient that define the complexity class that we study in this work. The two notions are: (i) bounded maximum local variation and (ii) bounded mean oscillation. We later argue that these two (local) assumptions are sufficient for obtaining upper complexity bounds. Both properties are defined for a fixed radius $r.$ %
We later discuss how this radius can be chosen or estimated. 
The motivation for the considered notions of local variability is illustrated in examples provided in Figure~\ref{fig:GradBMO-comparison}, as discussed in the introduction. %

\subsection{Bounded Maximum Local Variation}\label{sec:GradBMV}

Bounded maximum local variation requires that the subgradient of a function does not change much over small regions, although it is possible for each of the subgradients to have a large norm. Formally,
\begin{definition}[\gradBMV]\label{def:bounded-max-loc-var} 
    Given $r > 0,$ we say that a locally Lipschitz function $f: \rr^d \to \rr$ has bounded maximum local variation of the subgradient in norm $\|\cdot\|_2$ (is \gradBMV), if there exists a positive constant $\Llocr < \infty$ such that 
    \begin{equation}\notag
      (\forall\vx \in \rr^d)(\forall \vu \in \cb): \;\; \|\grad(\vx + r\vu) - \grad (\vx) \|_2 \leq \Llocr.
    \end{equation}
\end{definition}
Observe that if $f$ were (globally) $L$-Lipschitz continuous, then $\Llocr \leq 2 L.$ However, it is possible for $\Llocr$ to be much smaller than $L.$ For example, a univariate function $f$ equal to $\frac{1}{2}$ for $|x|<1$ and equal to $\frac{1}{2}x^2$ %
otherwise is clearly not Lipschitz continuous, but has bounded local variation of the (sub)gradient with $\Llocr = 1$ for $r < 2.$ (Observe that the ``local'' Lipschitz constant  as in \cite{cutkosky2023optimal,davis2022nearly} would be much larger in general, as the derivative of $f$ scales with $x$ for $|x|>1,$ which is bounded on bounded sets, but scales with the diameter of the set.) \markupadd{Further, this example function is not differentiable at $x = 1,$ thus it is clearly nonsmooth.}%

\subsection{Bounded Mean Oscillation and Smoothing}\label{sec:GradBMO}

Functions with Bounded Mean Oscillation (BMO functions) play an important role in harmonic analysis. They are formally defined as follows (see, e.g., \cite{john1961functions}).

\begin{definition}[Bounded Mean Oscillation]\label{def:BMO-functions}
    Let $f:\rr^d \to \rr$ be a function that is integrable on compact sets. 
    Then, $f$ is said to be a BMO function if
    \begin{equation}\label{eq:BMO-def}
        \|f\|_{\rm BMO} := \sup_{\vx \in \rr^d, r > 0} \frac{1}{\vol(\cB_r)}\int_{\cB_r}|f(\vx + \vu) - f_r(\vx)|\dd\vu < \infty,
    \end{equation}
    where $\vol$ denotes the volume and 
    \begin{equation}\label{eq:avg-fun-def}
        f_r(\vx) = \frac{1}{\vol(\cB_r)}\int_{\cB_r}f(\vx + \vu)\dd\vu.
    \end{equation}
    The (semi-)norm $\|\cdot\|_{\rm BMO}$ is referred to as the BMO norm.
\end{definition}

Note that under minimal assumptions (e.g., local Lipschitzness), $f_r$ is differentiable. 
The integral $\frac{1}{\vol(\cB_r)}\int_{\cB_r}|f(\vx + \vu) - f_r(\vx)|\dd\vu$ is known as the mean oscillation of $f$ over $\cbr$. 
The above definition is often stated for the unit ball w.r.t.\ $\|\cdot\| = \|\cdot\|_\infty,$ i.e., by defining BMO functions as functions with bounded oscillations $\frac{1}{\vol(\cB_r)}\int_{\cB_r}|f(\vx + \vu) - f_r(\vx)|\dd\vu$ over hypercubes. However, using the definition as ours is not uncommon in the literature and the definitions using different norms are all equivalent (though the value of the resulting BMO norms may differ) \cite{wiegerinck1997bmo,stein1993harmonic}. In this work, we focus on the Euclidean case, where $\|\cdot\| = \|\cdot\|_2.$ All bounded functions are BMO. 

\paragraph{Gradient BMO Functions}

The definition of BMO functions is not directly useful in our setting, for two reasons: (1) the bounded oscillation is defined with respect to the function value, whereas in our case it is the slope (or the subgradient)  whose changes with respect to small perturbations determine how close a function is to its smoothed approximation (see Figure~\ref{fig:GradBMO-comparison}); and (2) BMO is a global property of functions, whereas we are interested in small, local changes in the slope (or the subgradient). Accounting for these two issues, we introduce the following definition of (sub)gradient $r$-BMO functions.
\begin{definition}[\gradBMO]\label{def:BMO-grad-sphere}
    Let $f:\rr^d \to \rr$ be a locally Lipschitz function. %
    We say that $f$ is \gradBMO\ if there exist $r>0$ and $L_r < \infty$ such that 
    \begin{equation}\label{eq:r-BMO-grad-def}
        \|\nabla f\|_{{\rm BMO}, r} := \sup_{0 < \rho \leq r}\sup_{\vx \in \rr^d} \frac{1}{\vol(\cs_{\rho})}\int_{\cs_{\rho}}\|\grad(\vx + \vu) - \nabla f_r(\vx)\|_2\dd\vu \leq L_r,
    \end{equation}
    where $f_r$ is defined by \eqref{eq:avg-fun-def} and $\cs_{\rho} = \{\vx: \|\vx\|_2 = \rho\}$.  
    In this case, we also say that $f$ is \gradBMO\ with constant $L_r.$ 
    %
    %
    %
\end{definition}

\gradBMO\ property plays a role in local smoothing of a function. Intuitively, functions with lower \gradBMO\ constants (for the same $r$) are closer to their smooth approximations obtained using local smoothing (such as randomized smoothing over small balls used in this work). Our main insight is that this property, together with the \gradBMV\ from Definition~\ref{def:bounded-max-loc-var}, allows us to characterize complexity of nonsmooth optimization problem classes at a finer scale. For illustration, recall the two functions shown in Fig.~\ref{fig:GradBMO-comparison}. Both these functions are nonsmooth (in fact, both are piecewise-linear) with the same Lipschitz constant. However, the right function has a smaller \gradBMO\ constant $L_r$ for a sufficiently small radius indicated on the right subfigure. Even though both functions belong to the same class of nonsmooth Lipschitz functions, visually, the right functions is ``closer to being smooth,'' as the transitions between the linear pieces have less dramatic changes in the slope.  

\markupadd{For our results, we impose a bound on the oscillation at multiple scales, and that naturally introduces the supremum over $\rho$ in the range $(0,r]$. This requirement comes from the need to simultaneously bound $\frac{1}{\vol(\cs_{r})}\int_{\cs_{r}}\|\grad(\vx + \vu) - \nabla f_r(\vx)\|_2\dd\vu$ (in e.g., Lemma~\ref{lem:smoothness_BMO}) and $\frac{1}{\vol(\cb_{\rho})}\int_{\cb_{\rho}}\|\grad(\vx + \vu) - \nabla f_r(\vx)\|_2\dd\vu$ for $\rho \in (0, r]$ (in e.g., Lemma~\ref{lemma:generic-smooth-approx-err}). Both these quantities can be bounded by $\sup_{0 < \rho \leq r} \frac{1}{\vol(\cs_{\rho})}\int_{\cs_{\rho}}\|\grad(\vx + \vu) - \nabla f_r(\vx)\|_2\dd\vu$ as in \eqref{eq:r-BMO-grad-def}.}

\paragraph{Randomized Local Smoothing}

Observe that, given $r>0$, if we consider the uniform distribution $\unif(\cb_r)$ on the centered Euclidean ball of radius $r,$ $\cb_r = r\cb,$ then we can equivalently define the smoothed function $f_r$ from \eqref{eq:avg-fun-def} as 
\begin{equation}\label{eq:f_r-via-expectation}
\begin{aligned}
    f_r(\vx) &= \ee_{\vu \sim \unif(\cb_r)} [f(\vx + \vu)]
    = \ee_{\vu \sim \unif(\cb)} [f(\vx + r\vu)].
\end{aligned}
\end{equation}
This observation gives rise to the use of randomized smoothing, where we can obtain an unbiased estimate of the (sub)gradient of $f_r$ using one of the two following ideas. The first is simply using $\grad(\vx + \vu)$, where $\vu$ is drawn uniformly at random from $\cbr$: this results in an unbiased estimate by the dominated convergence theorem.  The second is $f(\vx + r\vu)\vu,$ where $\vu$ is drawn uniformly at random from the sphere of radius one. This is a valid unbiased estimate as a consequence of Stokes theorem, summarized in the lemma below. The proof of the lemma can be found in \cite[Chapter 9]{nemirovski:1983} and in \cite{flaxman2005online}, and is thus omitted for brevity. 

\begin{lemma}\label{lemma:stokes}
    Given $r > 0,$ 
    \begin{equation}\label{eq:stokes-estimate}
        \ee_{\vu \sim \unif(\cs)}[f(\vx + r\vu)\vu] = \frac{r}{d}\nabla f_r(\vx).
    \end{equation}
\end{lemma}

We first show that for \gradBMO\ convex functions with small constant $L_r$, the smoothed function $f_r$ is close to the original function $f$, which aligns well with our intuition from Fig.~\ref{fig:GradBMO-comparison}. 

\begin{lemma}\label{lemma:generic-smooth-approx-err}
    Let $f:\rr^d \to \rr$ be a locally Lipschitz function that is \gradBMO\ with  constant $L_r$ and let $f_r$ be defined by \eqref{eq:f_r-via-expectation}. Then, for all $\vx \in \rr^d,$
    \begin{equation}\notag
        f_r(\vx) - f(\vx) \leq L_r r.
    \end{equation}
    Additionally, if $f$ is convex, then $f_r(\vx) - f(\vx) \geq 0,$ $\forall \vx \in \rr^d.$
\end{lemma}
\begin{proof}
    The second claim (for convex functions) follows from Jensen's inequality. For the first claim, we start by using the definition of $f_r$ and the fundamental theorem of calculus (see Theorem \ref{thm:FTC}) applied to $f$ to conclude that
    \begin{align*}
        f_r(\vx) - f(\vx) &= \ee_{\vu \sim \unif(\cb)} [f(\vx + r\vu) - f(\vx)]\\
        &= \ee_{\vu \sim \unif(\cb)} \bigg[\int_{0}^1\innp{\gamma_f(\vx + tr\vu), r\vu}\dd t\bigg].
    \end{align*}
    Because $\unif(\cbr)$ is centrally symmetric, we have that $\ee_{\vu \in \unif(\cb)}[\innp{\vz, \vu}] = 0$ for any fixed vector $\vz \in \rr^d$. Hence, $\ee_{\vu \in \unif(\cb)}[\innp{\nabla f_r(\vx), \vu}] = 0$ and thus we can write
    \begin{align*}
        |f_r(\vx) - f(\vx)| &= r\bigg|\ee_{\vu \sim \unif(\cb)} \bigg[\int_0^1\innp{\grad(\vx + tr\vu) - \nabla f_r(\vx), \vu}\dd t\bigg]\bigg|\\
        &\stackrel{(i)}{\leq} r\ee_{\vu \sim \unif(\cb)} \bigg[\int_0^1\|\grad(\vx + tr\vu) - \nabla f_r(\vx)\|_2 \|\vu\|_2\dd t\bigg]\\
        &\stackrel{(ii)}{\leq} r\ee_{\vu \sim \unif(\cb)} \bigg[\int_0^1\|\grad(\vx + tr\vu) - \nabla f_r(\vx)\|_2 \dd t\bigg]\\
        &\stackrel{(iii)}{=} r\int_0^1\ee_{\vu \sim \unif(\cb)} \big[\|\grad(\vx + tr\vu) - \nabla f_r(\vx)\|_2\big] \dd t\\
        &\stackrel{(iv)}{\leq} rL_r,
    \end{align*}
    where ($i$) is by Jensen's inequality and Cauchy-Schwarz, ($ii$) is by  $\vu \in \cb$, so $ \|\vu\|_2 \leq 1,$ ($iii$) is by %
    Fubini's theorem, 
    and ($iv$) follows from Definition~\ref{def:BMO-grad-sphere}. 
\end{proof}

It is possible to obtain a tighter bound on the distance between $f$ and $f_r$ under an additional assumption about the subgradients of $f.$ This result is summarized in the following lemma and it will be particularly useful for obtaining near-dimension-independent convergence results in the parallel optimization setting.

\begin{lemma}\label{lemma:sparse-smooth-approx-err}
    Let $f$ be a locally Lipschitz function. Then for almost all $\vx\in\rr^d$
    \begin{equation}\notag
        f_r(\vx) - f(\vx) \leq r\,w(\partial_r f(\vx)),
    \end{equation}
where $w(K)=\mathbb{E}_{\vu\sim \unif{({\cal S})}}[\sup_{\vx_1, \vx_2\in K}\langle \vu,\vx_1 - \vx_2\rangle]$ denotes the mean width of a set $K$.
\end{lemma}

\begin{proof}
For a generic $\vx\in\rr^d$, we have that $f$ is differentiable at $\vx$. Hence, using the first theorem of calculus and the central symmetry of $\cbr$, we have
\begin{align*}
        f_r(\vx)-f(\vx) 
        &= \ee_{\vu \sim \unif(\cb_r)}[f(\vx+\vu)-f(\vx)]\\
    &= \ee_{\vu \sim \unif(\cb_r)}\bigg[\int_0^1\innp{\grad(\vx+t\vu), \vu}\dd t \bigg]\\
    &= \ee_{\vu \sim \unif(\cb_r)}\bigg[\int_0^1\innp{\grad(\vx+t\vu) - \grad(\vx), \vu}\dd t \bigg]. 
    \end{align*}
    Thus, letting $K=\partial_r f(\vx)$ we can further conclude that
    \begin{align}
        f_r(\vx)-f(\vx) &\leq \ee_{\vu \sim \unif(\cb_r)}\bigg[\int_0^1 \sup_{\vg_1, \vg_2 \in K}\innp{\vg_1 - \vg_2, \vu}\dd t \bigg] \notag\\
        &= r \ee_{\vu \sim \unif(\cb)}\Big[\sup_{\vg_1, \vg_2 \in K}\innp{\vg_1 - \vg_2, \vu}\Big]\notag\\
        &= r d\int_0^1 \mathbb{E}_{\vu\sim \unif({\cal S})}\Big[ \sup_{\vg_1, \vg_2 \in K}\innp{\vg_1 - \vg_2, \tau\vu} \Big] \tau^{d-1} \dd\tau \label{eqn:polar_coords}\\
        &= r  w(K)  d\int_0^1 \tau^d \dd\tau \notag\\ &= \frac{d}{d+1} rw(K),\notag%
    \end{align}
where in \eqref{eqn:polar_coords} we used integration by polar coordinates.
\end{proof}

\begin{remark}\label{rem:bounded-Gaussian-width}
For discussions on the mean width, and the closely related Gaussian width, we refer the interested reader to \cite[Section 7.5]{Vershynin:2018}. We provide some useful examples of mean width bounds from this reference: %
\begin{enumerate}
\item[(i)] Euclidean ball: $w(\cb)=1$.
\item[(ii)] Cube: $w([-1,+1]^d)=\Theta(\sqrt{d})$.
\item[(iii)] Polytopes: If $K$ is a polytope with $k$ vertices, then $w(K)=O\Big(\mbox{diam}(K)\sqrt{\frac{\log(k)}{d}}\Big)$. 
\end{enumerate}
The last example is particularly important. Many problems of interest in convex optimization can be formulated as (or approximated by) the maximum of finitely-many affine functions. In that case, the $1/\sqrt{d}$ factor in the mean width bound provides a much more benign approximation than the worst-case bound for Lipschitz functions, corresponding to example (i).
\end{remark}

We now argue about the smoothness of the smoothed function $f_r$. 

\begin{lemma} \label{lem:smoothness_BMO}
    Let $f:\rr^d \to \rr$ be a \gradBMV\  
    function with constant $\Llocr$ and \gradBMO\ 
    with constant $L_r,$ where both constants are defined w.r.t.~the same fixed radius $r > 0.$ Let $f_r$ be defined by \eqref{eq:avg-fun-def}. Then, for all $\vx, \vy \in \rr^d,$
    \begin{equation}\notag
        \|\nabla f_r(\vx) - \nabla f_r(\vy)\|_2 \leq \min\Big\{\frac{L_r d}{r}, \, \markupadd{\sqrt{\frac{\pi}{2}}\frac{\Llocr \sqrt{d}}{r}}\markupdelete{O\Big(\frac{\Llocr \sqrt{d}}{r}\Big)}\Big\}\|\vx - \vy\|_2\markupdelete{,}\markupadd{.}
    \end{equation}
    \markupdelete{where $O(\cdot)$ only hides an absolute constant. }
\end{lemma}
\begin{proof}
    For the first bound in the min, we use 
    Lemma~\ref{lemma:stokes}, as follows. 
    \begin{align*}
        \|\nabla f_r(\vx) - \nabla f_r(\vy)\|_2 &= \frac{d}{r}\big\|\ee_{\vu \sim \unif(\cs)}\big[\big(f(\vx + r\vu) - f(\vy + r\vu)\big)\vu\big]\big\|_2\\
        &= \frac{d}{r}\bigg\|\ee_{\vu \sim \unif(\cs)}\bigg[\int_0^1\innp{\grad(\vx + r\vu + t(\vy - \vx)), \vy - \vx} \dd t\vu\bigg]\bigg\|_2,
    \end{align*}
    where we have used the first theorem of calculus. Further, $\ee_{\vu \sim \unif(\cs)} [c\vu] = \zeros,$ for any constant $c,$ as $\vu$ is centrally symmetric. Hence, we can further write 
    \begin{align*}
       &\; \|\nabla  f_r (\vx) - \nabla f_r(\vy)\|_2 \\
       =\; & \frac{d}{r}\bigg\|\ee_{\vu \sim \unif(\cs)}\bigg[\int_0^1\innp{\grad(\vx + r\vu + t(\vy - \vx)) - \nabla f_r(\vx + t(\vy - \vx)), \vy - \vx} \dd t\vu\bigg]\bigg\|_2\\
        \leq\; & \frac{d}{r}\ee_{\vu \sim \unif(\cs)}\bigg[\bigg\|\int_0^1\innp{\grad(\vx + r\vu + t(\vy - \vx)) - \nabla f_r(\vx + t(\vy - \vx)), \vy - \vx} \dd t\vu\bigg\|_2\bigg]\\
        \leq\; & \frac{d}{r}\ee_{\vu \sim \unif(\cs)}\bigg[\int_0^1|\innp{\grad(\vx + r\vu + t(\vy - \vx)) - \nabla f_r(\vx + t(\vy - \vx)), \vy - \vx}| \dd t\|\vu\|_2\bigg]\\
        \leq \; &\frac{d}{r}\ee_{\vu \sim \unif(\cs)}\bigg[\int_0^1\|\grad(\vx + r\vu + t(\vy - \vx)) - \nabla f_r(\vx + t(\vy - \vx))\|_2\|\vy - \vx\|_2 \dd t\bigg]\\
        =\; & \frac{d \|\vy - \vx\|_2}{r}\ee_{\vu \sim \unif(\cs)}\bigg[\int_0^1\|\grad(\vx + r\vu + t(\vy - \vx)) - \nabla f_r(\vx + t(\vy - \vx)))\|_2 \dd t\bigg]\\
        \leq \; & \frac{L_r d \|\vy - \vx\|_2}{r},
    \end{align*}
    where we have used Jensen's inequality (twice), $\|\vu\|_2 = 1$ for $\vu \in \cs$, Cauchy-Schwarz inequality, and Definition~\ref{def:BMO-grad-sphere}.

    For the second bound, we use the following sequence of inequalities with $\vz := \vy - \vx$:
    \begin{align*}
    &\|\nabla f_r(\vx)-\nabla f_r(\vy)\|_2\\
    =\; & \frac{d}{r}\Big\|\ee_{\vu\sim \unif(\cs) }[(f(\vx+r\vu)-f(\vy+r\vu))\vu] \Big\|_2\\
    =\; &\frac{d}{r}\sup_{\|\vv\|_2\leq 1}\innp{ \vv,\mathbb{E}_{\vu\sim \unif(\cs)}[(f(\vx+r\vu)-f(\vy+r\vu))\vu] }\\
    =\; & \frac{d}{r}\sup_{\|\vv\|_2 \leq 1}\innp{ \vv, \mathbb{E}_{\vu\sim \unif(\cs)}\bigg[\int_0^1\innp{\grad(\vx + r\vu + t\vz), \vz} \dd t\vu\bigg] }\\
    =\; & \frac{d}{r}\sup_{\|\vv\|_2\leq 1}\innp{ \vv, \mathbb{E}_{\vu\sim \unif(\cs)}\bigg[\int_0^1\innp{\grad(\vx + r\vu + t\vz) - \grad(\vx + t\vz)), \vz} \dd t\vu\bigg] }\\
    =\; & \frac{d}{r}\sup_{\|\vv\|_2\leq 1}\mathbb{E}_{\vu\sim \unif(\cs)}\bigg[\innp{ \vv, \vu} \int_0^1\innp{\grad(\vx + r\vu + t\vz) - \grad(\vx + t\vz)), \vz} \dd t\bigg] \\
    \leq\; & \frac{d}{r}\sup_{\|\vv\|_2\leq 1}\mathbb{E}_{\vu\sim \unif(\cs)}\big[|\innp{ \vv, \vu}|\big] \sup_{\|\vw\|_2\leq 1}\bigg|\int_0^1\innp{\grad(\vx + r\vw + t\vz) - \grad(\vx + t\vz)), \vz} \dd t\bigg| \\
    \leq\; & \frac{\Llocr d}{r}\|\vy - \vx\|_2\sup_{\|\vv\|_2\leq 1}\mathbb{E}_{\vu\sim \unif(\cs)}\big[|\innp{ \vv, \vu}|\big],
\end{align*}
where we have used Definition \ref{def:bounded-max-loc-var} in the last line. To complete the proof, it remains to use 
\begin{align*}
 \sup_{\|\vv\|_2\leq 1}\mathbb{E}_{\vu\sim \unif(\cs)}\big[|\innp{ \vv, \vu}|\big]
=\;& \markupadd{\sup_{\|\vv\|_2\leq 1}\int_0^{\infty} \mathbb{P}[|\langle \vv,\vu\rangle|>t] \dd t}\\
\leq\; & \markupadd{\int_0^{+\infty}\exp(-dt^2/2)\dd t =\sqrt{\frac{\pi}{2d}}}
\markupdelete{=O\Big(\frac{1}{\sqrt{d}}\Big)},
\end{align*}
\markupadd{where we have used $\Pr[\innp{ \vv, \vu} \geq c] \leq e^{-\frac{d c^2}{2}}$,} which holds \markupadd{for all $c \in [0, 1),$} by the concentration of measure on a (unit) sphere; see, e.g., \cite[Lemma 2.2]{Ball:1997}. %
\end{proof}

\subsection{A Discussion of \gradBMV\ and \gradBMO\ Classes}\label{sec:examples}

We now provide some examples that illustrate how classes of \gradBMV\ and \gradBMO\ functions compare to each other and to classical classes of objective functions studied in the optimization literature. First, based on the definition of constants $L_r$ and $\Llocr$ defining the \gradBMO\ and \gradBMV\ classes, it is immediate that
\begin{equation}\label{eq:Lr<Llocr}
    L_r \leq \widehat{L}_{2r}.
\end{equation}
As it turns out, it is possible for $L_r$ to be much smaller than $\Llocr,$ as illustrated below. 

\markupadd{\begin{example}\label{ex:Lr-vs-Llocr-worst-case}
    Consider the following function $f:\rr^d\to\rr$ for $K = \left\lceil \sqrt{\frac{d}{2\ln(2d)}}\right\rceil:$
\begin{equation}\notag
    f(\vx) = \max_{0 \leq i \leq K} \Big\{\frac{i}{K}x_1 - \frac{(i-1)i}{K^2}\Big\}.
\end{equation}
This is a piecewise-linear function dependent only on $x_1,$ with subgradient being the zero vector for $x_1 \leq 0,$ and then gradually increasing as $\frac{i}{K} \ve_1$ as $x_1$ is increased between 0 and 1 in $1/K$ increments. As a consequence, $\grad(\vx) = \sum_{i=1}^K \frac{i}{K} \ones\{x_1 \in (\frac{i-1}{K}, \frac{i}{K}]\}\ve_1 + \ones\{x_1 \geq 1\}\ve_1$ is a subgradient of $f$ at $\vx$ for all $\vx \in \rr^d,$ where $\ones\{\cdot\}$ is one if its argument is true and  zero otherwise, while $\ve_1$ denotes the first standard basis vector. Taking $r = 1,$ it is immediate that $\Llocr = 1.$ To bound $L_r$ above, it is evident that only vectors with $x_1\in[0,1]$ can attain the supremum that defines $L_r$. Suppose that $\frac{i-1}{K}<x_1\leq \frac{i}{K}$, for $i \in \{1, \dots, K\}$. Let $\vu\sim \unif(\cs_{\rho})$ for $0<\rho\leq 1$. 
Using concentration of measure on the sphere, we have that with probability at least $1-1/d$, $|u_1|\leq 1/K$. As a consequence,  under the same event, $x_1 + u_1$ can only reach up to two (neighboring) linear pieces defining $f(\vx)$. Without loss of generality, in the worst case, 
$\frac{i-2}{K}< x_1 + u_1 <\frac{i+1}{K}$. 
This implies that $\nabla f_r(\vx)\in[\frac{i-2}{K}, \frac{i + 1}{K}]$, and that with probability $1-1/d$, $\gamma_f(\vx+\vu)\in \{\frac{i-1}{K}, \frac{i}{K},\frac{i+1}{K}\}$, thus
\[ \frac{1}{\vol(\cs_{\rho})}\int_{\cs_{\rho}}\|\gamma_f(\vx+\vu)-\nabla f_r(\vx)\|_2 \dd\vu \leq \frac{1}{d}+ O\Big(\frac{1}{K}\Big) =O\Big(\frac{1}{K}\Big)=O\Big(\sqrt{\frac{\ln d}{d}}\Big). \]
Taking $d\to+\infty$, we conclude that the ratio between $\hat L_r$ and $L_r$ can be arbitrarily large.
\end{example}}

\markupadd{Additionally, if we look at \emph{local} values of $L_r$ and $\Llocr$ (in a ball of radius $r$  around a point $\vx$ that we fix)---denoted by $\Llocr(\vx)$ and $L_r(\vx)$ to disambiguate from their worst-case values defined earlier in this section---it is possible for $L_r(\vx)$ to be even exponentially smaller than $\Llocr(\vx),$ in terms of the dimension. This is illustrated in Example \ref{ex:Lr-vs-Llocr} below. The provided observation makes the possibility of developing algorithms that are adaptive to local values of $L_r$ particularly appealing as a direction for future research.}

\begin{example}\label{ex:Lr-vs-Llocr}
    Consider the following function $f:\rr^d\to\rr$ for $c \in [1, \sqrt{d - 1})$:
    \begin{equation}\label{eq:f-for-Lr-vs-Llocr}
        f(\vx) = \begin{cases}
            0, &\text{ if } |x_1| \leq \frac{c}{\sqrt{d-1}}\\
            |x_1| - \frac{c}{\sqrt{d-1}}, &\text{ if } |x_1| > \frac{c}{\sqrt{d-1}}
        \end{cases}.
    \end{equation}
    This function is convex and 1-Lipschitz continuous. It is differentiable everywhere except at {the set $\{|x_1|=\frac{c}{\sqrt{d-1}}\}$}, with its subdifferential given by
    \begin{equation}\label{eq:grad-f-for-Lr-vs-Llocr}
        \partial f(\vx) = \begin{cases}
            \{\zeros\}, &\text{ if } |x_1| < \frac{c}{\sqrt{d-1}}\\
            \{\mathrm{sign}(x_1)\ve_1\}, &\text{ if } |x_1| > \frac{c}{\sqrt{d-1}}\\
            \mathrm{conv}(\zeros, \mathrm{sign}(x_1)\ve_1), &\text{ if } |x_1| = \frac{c}{\sqrt{d-1}}
        \end{cases},
    \end{equation}
    where $\mathrm{sign}(x_1)$ is equal to one if $x_1 \geq 0$ and is equal to $-1$ otherwise. 
    For $r \geq 2\frac{c}{\sqrt{d-1}}$ \markupadd{and $\vx = \zeros,$ assuming $\grad(\vx) = \arg\sup_{\vg \in \partial f(\vx)}\|\vg\|_2,$} we have that $\Llocr\markupadd{(\zeros)} = 2$, which immediately follows from \eqref{eq:grad-f-for-Lr-vs-Llocr}\markupadd{, since $\|\grad((r/2) \ve_1) - \grad(-(r/2)\ve_1)\|_2 = 2$}. On the other hand, \markupdelete{the \gradBMO\ constant} $L_r\markupadd{(\zeros)}$ is determined by the average subgradient variation around $\vx = \zeros,$ where, by symmetry, $\nabla f_r(\vx) = 0.$ Assuming $\frac{c}{\sqrt{d-1}} \leq \frac{1}{2}$ and taking $r=1,$ we thus get that
    \begin{equation}\notag
        L_r\markupadd{(\zeros)} = \sup_{0 < \rho \leq 1}\int_{\cs_{\rho}}\|\grad(\vx + \vu)\|_2\dd\vu = \int_{\cs_1}\|\grad(\vu)\|_2\dd\vu = \int_{\cs_1 \cap \{|u_1| > \frac{c}{\sqrt{d-1}}\}}1\dd\vu.
    \end{equation}
    Using an adaption of \cite[Theorem 2.7]{blum2020foundations} from the unit ball to the unit sphere, it is possible to show that the above integral is bounded by $\frac{2}{c}e^{-c^2/2}.$ Thus, taking, e.g., $c = \frac{\sqrt{d-1}}{2}$ and increasing the dimension $d,$ we can make $L_r\markupadd{(\zeros) \leq \frac{4}{\sqrt{d-1}}e^{-\frac{d-1}{8}}}$ arbitrarily small. 
\end{example}

Because $M$-Lipschitz continuous functions can be equivalently defined as functions whose subgradient is uniformly bounded by $M,$ it is immediate that $L_r \leq \Llocr \leq 2M$ for any $r > 0.$ The latter inequality is tight in general, as is apparent by considering the univariate function $f(x) = M|x|.$ 

Classes of $(M, \kappa)$-weakly smooth functions for $\kappa \in [0, 1]$ are also captured by classes of \gradBMV and \gradBMO\ functions. In particular, $(M, \kappa)$-weakly smooth functions are defined as continuously differentiable functions with H\"{o}lder-continuous gradient, satisfying
\begin{equation}\label{eq:weakly-smooth-def}
    \|\nabla f(\vx) - \nabla f(\vy)\|_2 \leq M\|\vx - \vy\|_2^\kappa, \quad \forall \vx, \vy \in \rr^d.
\end{equation}
It is immediate from this definition that $L_r \leq \Llocr \leq M r^\kappa,$ $\forall r > 0.$ 

It is possible that a function is \gradBMV\ (and thus also \gradBMO) but neither Lipschitz-continuous nor weakly smooth, for any finite $M$ and $\kappa \in [0, 1].$ This is illustrated by the following example, which extends the univariate example from Section~\ref{sec:GradBMV}.

\begin{example}\label{ex:nonsmooth-quadratic-growth}
    Consider the function $f: \rr^d \to \rr$ defined by
        \begin{equation}\label{eq:nonsmooth-quadratic-growth}
            f(\vx) = \begin{cases} 0, &\text{ if } \|\vx\|_2 \leq 1,\\
            \frac{1}{2}\|\vx\|_2^2 - \frac{1}{2}, &\text{ otherwise}
            \end{cases}.
        \end{equation}
    This is a continuous function whose gradient is discontinuous on the sphere $\|\vx\|_2 = 1.$ In more detail, within the unit ball $\|\vx\|_2,$ $\nabla f(\vx) = 0,$ while outside the unit ball $\nabla f(\vx) = \vx.$ Thus, for $r \in (0, 1],$ $\Llocr = \sup_{\vx \in \rr^d, \vu \in \cb}\|\nabla f(\vx + r\vu) - \nabla f(\vx)\|_2 = 1.$ On the other hand, this function is neither weakly smooth (as its gradient is not continuous) nor globally Lipschitz continuous (as for $\|\vx\|_2 > 1,$ we have $\|\nabla f(\vx)\|_2 = \|\vx\|_2,$ which is unbounded). 
\end{example}

Another interesting consequence of Example~\ref{ex:nonsmooth-quadratic-growth} is that, unlike Lipschitz continuity, \gradBMV\ does not preclude quadratic growth of a function. This property appears particularly useful for the study of complexity of nonsmooth optimization under local error bound conditions \cite{pang1997error}, \cite[Chapter 8]{cui2021modern}, which can enable linear convergence of algorithms; see \cite{davis2022nearly} for one such example.  

\subsection{\markupadd{Structural Results}}

\markupadd{In this subsection, we derive structural results that are particularly useful for the analysis of standard first-order methods. We first show that \gradBMV\ functions can be characterized as being ``approximately smooth,'' in the sense that they behave as smooth functions between points further apart than $r.$ Precise statements formalizing this geometric intuition are provided in the following two lemmas.}

\markupadd{\begin{lemma}\label{lem:upper-quadratic-approx}
    Let $f:\rr^d \to \rr$ be \gradBMV\ at some radius $r_0 > 0.$ Then for all $r \in (0, r_0]$, for \gradBMV\ parameter $\Llocr$ associated with $r,$ the following holds:
    \begin{equation}\notag
        f(\vy) - f(\vx) - \innp{\grad(\vx), \vy - \vx} \leq \begin{cases}
            \frac{\Llocr}{2r}\|\vy - \vx\|_2^2, \text{ if } \|\vy - \vx\|_2 > r,\\
            \Llocr\|\vy - \vx\|_2, \text{ if } \|\vy - \vx\|_2 \leq r.
        \end{cases}
    \end{equation}
\end{lemma}}
\begin{proof}
\markupadd{The proof is based on interpreting the subgradient of $f$ as an inexact oracle for a smooth function. While this is a known idea in convex optimization \cite{nesterov2015universal,Devolder:2014}, we proceed differently from previous works when the points of interest lie further apart. Here we partition the line segment joining these two points and apply the \gradBMV\ property in each of these intervals (see Case 2 below). Aggregating these bounds then provides a sharper quadratic upper bound. }

    \markupadd{The definition of \gradBMV\ functions requires that the function is locally Lipschitz, thus FTC (Theorem \ref{thm:FTC}) applies, by which
    \begin{equation}\label{eq:approx-smooth-ftc}
        f(\vy) - f(\vx) - \innp{\grad(\vx), \vy - \vx} = \int_0^1 \innp{\gamma(\vx + t(\vy - \vx)) - \gamma(\vx), \vy - \vx}\dd t. 
    \end{equation}
    It is immediate from the definition of \gradBMV\ that if a function is \gradBMV\ at radius $r_0,$ then it is \gradBMV\ at any radius $r \in (0, r_0],$ possibly with a smaller value of the associated parameter $\Llocr.$ Fix any $r \in (0, r_0]$ and consider the following two possible cases.}

    \markupadd{\noindent\textbf{Case 1: $\|\vy - \vx\|_2 \leq r$.} Then, by the definition of \gradBMV\ functions, it must be $\|\gamma(\vx + t(\vy - \vx)) - \gamma(\vx)\|_2 \leq \Llocr$ for any $t \in [0, 1].$ Thus, bounding the inner product on the right-hand side of \eqref{eq:approx-smooth-ftc} using Cauchy-Schwarz inequality and  $\|\gamma(\vx + t(\vy - \vx)) - \gamma(\vx)\|_2 \leq \Llocr$, and integrating, we get that
    \begin{equation}\label{eq:approx-smooth-ltr}
        f(\vy) - f(\vx) - \innp{\grad(\vx), \vy - \vx} \leq \Llocr \|\vy - \vx\|_2.
    \end{equation}
    }

    \markupadd{\noindent\textbf{Case 2: $\|\vy - \vx\|_2 > r$.} Let $m = \lfloor \frac{\|\vy - \vx\|_2}{r} \rfloor + 1$. Then, we have
\begin{align*}
   \; & \int_{0}^1 \innp{\grad(\vx + t(\vy - \vx)) - \grad(\vx), \vy - \vx} \dd t\\
   =\; & \sum_{i = 1}^m \int_{(i-1)/m}^{i/ m} \innp{\grad(\vx + t(\vy - \vx)) - \grad(\vx), \vy - \vx} \dd t. 
\end{align*}
We now bound each of the integrals as follows:
\begin{align*}
   &\; \int_{(i-1)/m}^{i/ m} \innp{\grad(\vx + t(\vy - \vx)) - \grad(\vx), \vy - \vx} \dd t \\
   =\; & \int_{(i-1)/m}^{i/ m} \innp{\grad(\vx + t(\vy - \vx)) - \grad(\vx + (i-1)/m (\vy - \vx)), \vy - \vx} \dd t \\
   &+ \frac{1}{m}\sum_{j= 1}^{i-1} \innp{\grad(\vx + j/m (\vy - \vx)) - \grad(\vx_j + (j-1)/m (\vy - \vx)), \vy - \vx}\\
   &\; \leq \frac{1}{m}\Llocr \|\vy - \vx\|_2 + \frac{i-1}{m} \Llocr \|\vy - \vx\|_2 = \frac{i}{m} \Llocr\|\vy - \vx\|_2.
\end{align*}
Now summing over $i$ and plugging back into \eqref{eq:approx-smooth-ftc}, we finally get
\begin{align}
    f(\vy) - f(\vx) - \innp{\grad(\vx), \vy - \vx} &\leq \Llocr \|\vy - \vx\|_2 \sum_{i=1}^m \frac{i}{m}\notag\\
    &\leq \Llocr \frac{m-1}{2}\|\vy - \vx\|_2 \notag\\
    &\leq \frac{\Llocr}{2 r}\|\vy - \vx\|_2^2,\label{eq:approx-smooth-gtr}
\end{align}
where in the last inequality we used $m -1 = \lfloor \frac{\|\vy - \vx\|_2}{r} \rfloor \leq \frac{\|\vy - \vx\|_2}{r},$ by our choice of $m.$}

\markupadd{To complete the proof, it remains to combine \eqref{eq:approx-smooth-ltr} and \eqref{eq:approx-smooth-gtr}.}
\end{proof}

\markupadd{Another useful inequality that is a consequence of Lemma \ref{lem:upper-quadratic-approx} is akin to interpolation inequalities, which are used to characterize the class of smooth convex functions; see, for example \cite{taylor2017convex}. It is provided in the following lemma.}

\markupadd{\begin{lemma}
    Let $f:\rr^d \to \rr$ be \gradBMV\ at some radius $r_0 > 0.$ Then for all $r \in (0, r_0]$, for \gradBMV\ parameter $\Llocr$ associated with $r,$ the following holds: for all $\vx, \vy \in \rr^d$ such that $\|\grad(\vy) - \grad(\vx)\|_2 > \Llocr,$
    \begin{equation}\notag
        \frac{r}{2\Llocr}\|\grad(\vy) - \grad(\vx)\|_2^2 \leq f(\vy) - f(\vx) - \innp{\grad(\vx), \vy - \vx}. 
    \end{equation}
\end{lemma}}
\begin{proof}
    \markupadd{Following a standard approach for proving interpolation inequalities of this type (see, e.g., \cite{taylor2017convex}), fix $\vx \in \rr^d$ and consider the function $h_\vx$ defined by $h_{\vx}(\vy):= f(\vy) - \innp{\grad(\vx), \vy}.$ It is not hard to verify that this function is convex, \gradBMV\ for all $r \in (0, r_0],$ and minimized by $\vx.$ Fix any $r \in (0, r_0].$ Suppose that $\|\gamma_{h_{\vx}}(\vy)\|_2 = \|\grad(\vy) - \grad(\vx)\|_2 > \Llocr.$ Then $\|\vy - (\vy - \frac{r}{\Llocr}\gamma_{h_{\vx}}(\vy))\|_2 > r$. Thus, using $h_\vx(\vx) \leq h_{\vx}(\vy'),$ for all $\vy' \in \rr^d,$ and applying Lemma \ref{lem:upper-quadratic-approx}, we get
    \begin{align*}
        h_{\vx}(\vx) - h_\vx(\vy) &\leq h_\vx(\vy - \frac{r}{\Llocr}\gamma_{h_{\vx}}(\vy)) - h(\vy)\\
        &\leq - \frac{r}{\Llocr} \innp{\gamma_{h_{\vx}}(\vy), \gamma_{h_{\vx}}(\vy)} + \frac{\Llocr}{2r}\Big\|\frac{r}{\Llocr} \gamma_{h_{\vx}}(\vy)\Big\|_2^2\\
        &= - \frac{r}{2\Llocr}\|\gamma_{h_{\vx}}(\vy)\|_2^2.
    \end{align*}
    To complete the proof, it remains to plug the definition of $h_\vx$ into the above inequality and simplify.
    }
\end{proof}

\markupadd{Finally, to make use of approaches based on randomized smoothing, it is important to bound the variance of stochastic gradient estimates $\grad(\vx + r \vu),$ $\vu \sim \unif(\cB),$ of $\nabla f_r,$ where we recall the smoothed function $f_r$ was defined in \eqref{eq:avg-fun-def}, with an equivalent definition using expectations provided in \eqref{eq:f_r-via-expectation}.}

\markupadd{\begin{lemma}\label{lem:variance-bnd}
    Let $f:\rr^d \to \rr$ be \gradBMO\ at some radius $r > 0$ with parameter $L_r$ and \gradBMV\ at radius $2r$ with parameter $\widehat{L}_{2r}$. Then, for all $\vx \in \rr^d,$
    \begin{equation}\notag
        \ee_{\vu \sim \unif(\cb)}\big[\|\grad(\vx + r\vu) - \nabla f_r(\vx)\|_2^2\big] \leq L_r \widehat{L}_{2r}.
    \end{equation}
\end{lemma}}
\begin{proof}
    \markupadd{First, observe that for all $\vu \in \cb,$ $\|\grad(\vx + r\vu) - \nabla f_r(\vx)\|_2 \leq \widehat{L}_{2r},$ as a consequence of the definition of $f_r$, Jensen's inequality, and the definition of $\widehat{L}_{2r}.$ Thus,
    \begin{align*}
        &\; \ee_{\vu \sim \unif(\cb)}\big[\|\grad(\vx + r\vu) - \nabla f_r(\vx)\|_2^2\big]\\
         \leq\; & \widehat{L}_{2r} \ee_{\vu \sim \unif(\cb)}\big[\|\grad(\vx + r\vu) - \nabla f_r(\vx)\|_2\big]\\
         \leq\; & \widehat{L}_{2r} d\int_0^1 \mathbb{E}_{\vu\sim \unif({\cal S})}\big[\|\grad(\vx + \tau r \vu) - \nabla f_r(\vx)\|_2\big] \tau^{d-1} \dd\tau\\
         \leq\; & \widehat{L}_{2r} d\int_0^1 L_r \tau^{d-1} \dd\tau\\
         =\; & \widehat{L}_{2r} \frac{d}{d+1}L_r \leq \widehat{L}_{2r} L_r,
    \end{align*}
    where in the second inequality we used integration over polar coordinates and in the third we used $\mathbb{E}_{\vu\sim \unif({\cal S})}\big[\|\grad(\vx + \tau r \vu) - \nabla f_r(\vx)\|_2 \leq L_r,$ which holds for all $\tau \in [0, 1],$ by the definition of $L_r.$
    }
\end{proof}

\section{Optimization under Bounded Local Variation of the Subgradient}\label{sec:algorithms}

In this section, we discuss how to optimize functions with bounded local variation of the subgradient. We first provide bounds for convex optimization obtained using randomized smoothing and stochastic optimization methods applied to the smoothed function $f_r.$ We then provide alternative bounds based on Goldstein's method.

Throughout this section, we assume that there exists some finite radius $r_0$ for which the function $f$ is a \gradBMV\ function  (and thus is also a \gradBMO\ function). It is immediate by our definitions of \gradBMV\ and \gradBMO\ properties that both hold for any $r \in (0, r_0],$ possibly with smaller constants. For simplicity and as is standard, we  assume that $f$ is minimized by some $\vx_*$ in convex settings and bounded below by some $f_*$ in nonconvex settings. 

\subsection{Convex Optimization}

We begin this section by reviewing algorithm AGD+  \cite{cohen2018acceleration} (see also the ``method of similar triangles'' in \cite{gasnikov2018universal}), which we subsequently use to obtain convergence bounds under the \gradBMV\ property. For completeness, we provide full details of the analysis of AGD+, stated in slightly different terms than the original analysis from \cite{cohen2018acceleration} to make the application of those results more direct and suitable to our setting. 

\subsubsection{AGD+ and its Analysis}

Iterates of AGD+ applied to Euclidean, projection-based settings and for an arbitrary  estimate $\vg_k$ of $\grad(\vx_k)$ are defined by 
\begin{equation}\label{eq:AGD+}\tag{AGD+}
    \begin{aligned}
        \vx_k &= \frac{A_{k-1}}{A_k}\vy_{k-1} + \frac{a_k}{A_k} \proj_{\cx}(\vz_{k-1})\\
        \vz_k &= \vz_{k-1} - a_k \vg_k \\
        \vy_k &= \frac{A_{k-1}}{A_k}\vy_{k-1} + \frac{a_k}{A_k} \proj_{\cx}(\vz_{k}). 
    \end{aligned}
\end{equation}
We recall that $\vx_0 \in \cx$ is arbitrary and at initialization $\vy_0 = \proj_{\cx}(\vz_0),$ $\vz_0 = \vx_0 - a_0 \vg_0$. Recall also that $A_k = \sum_{i=0}^k a_i,$ where $a_i$, $i \geq 0,$ are positive step sizes. Denote $\vv_i := \proj_{\cx}(\vz_i).$ 

The analysis that we use here is slightly different than those in \cite{thegaptechnique, cohen2018acceleration}, in that we make the ``sources of error'' that constrain the convergence rate of the algorithm more explicit and suitable to our discussion in Section~\ref{sec:algorithms}. 
It is based on the approximate gap technique \cite{thegaptechnique}, which bounds a gap estimate $G_k(\vw) \geq f(\vy_k) - f(\vw)$, for $\vw \in \cx$. In particular, the argument constructs a lower bound on $f(\vw)$, $L_k(\vw) \leq f(\vw)$, and then bounds ``error terms'' $E_k,$ which satisfy $E_0 \geq A_0 G_0 - \frac{1}{2}\|\vw - \vx_0\|_2^2$ and  $E_k \geq A_{k}G_k - A_{k-1} G_{k-1}.$ Then the optimality gap is simply bounded using 
\begin{equation}\label{eq:gap-generic-bnd}
    f(\vy_k) - f(\vw) \leq G_k(\vw) \leq \frac{(1/2)\|\vw - \vx_0\|_2^2 + \sum_{i=0}^k E_i}{A_k}
\end{equation}
and we draw inferences about the convergence by choosing the sequence $A_k$ to ensure the right-hand side of \eqref{eq:gap-generic-bnd} decays as fast as possible with $k.$ %

Throughout this section, we let $\vw \in \cx$ be arbitrary but fixed. In particular, $\vw$ can be chosen as a minimizer of $f$, in which case \eqref{eq:gap-generic-bnd} bounds the optimality gap. For notational simplicity, we write $G_k, L_k$ to mean $G_k(\vw),$ $L_k(\vw),$ as the context is clear. 

We now define the ``error terms'' $E_k,$ $k \geq 0,$ that appear in the analysis. There are three main sources of error: (i) $E_{k}^s$ corresponding to smoothness of $f$ (less than or equal to zero if $f$ is smooth and step sizes are appropriately chosen), (ii) $E_k^b,$ corresponding to the bias of gradient estimates $\vg_k,$ and (iii) $E_k^v$ related to the variance of gradient estimates $\vg_k.$ We define them as follows and note that they are directly obtained from the subsequent analysis.
\begin{equation}\label{eq:error-terms-def}
    \begin{aligned}
        E_k^s &:= A_k\Big(f(\vy_k) - f(\vx_k) - \innp{\grad(\vx_k), \vy_k - \vx_k}  - \frac{A_k}{2{a_k}^2}\|\vy_k - \vx_k\|_2^2\Big)\\
        E_k^b &= E_k^b(\vw) := a_k \innp{\vg_k - \grad(\vx_k), \vw - \vx_k}\\
        E_k^v &:= a_k \innp{\grad(\vx_k)- \vg_k, \vv_k - \vx_k},\\
        E_k &:= E_k^s + E_k^b + E_k^v.
    \end{aligned}
\end{equation}

We begin by constructing the gap estimates $G_k,$ in the following proposition.

\begin{proposition}\label{prop:G_k-construction}
    Let $f: \rr^d \to \rr$ be a proper convex continuous function and let $\vx_i, \vy_i, \vz_i$ for $i \geq 0$ be the iterates of \eqref{eq:AGD+}.  %
    Let $\vw \in \cx$ be arbitrary. 
    Then for all $k \geq 0,$ $f(\vy_k) - f(\vw) \leq G_k,$ where
    \begin{equation}\label{eq:G_k-def}
    \begin{aligned}
        G_k :=\; & f(\vy_k) - \frac{1}{A_k}\sum_{i=0}^k a_i \big(f(\vx_i) + \innp{\vg_i, \vv_k - \vx_i} \big) + \frac{1}{A_k}\sum_{i=0}^k E_i^b\\
        &- \frac{1}{2A_k}\|\vv_k - \vx_0\|_2^2 + \frac{1}{2A_k}\|\vw - \vx_0\|_2^2
    \end{aligned}
    \end{equation}
\end{proposition}
\begin{proof}
    Define $G_k := f(\vy_k) - L_k,$ where $L_k \leq f(\vw).$ To carry out the proof, we then just need to construct $L_k$ that bounds $f(\vw)$ below and agrees with the expression from the statement of the proposition. 
    First, by convexity of $f$ and $A_k = \sum_{i=0}^k a_i,$ we have
    \begin{align*}
        f(\vw) \geq \;& \frac{1}{A_k}\sum_{i=0}^k a_i \big(f(\vx_i) + \innp{\grad(\vx_i), \vw - \vx_i}\big)\\
        =\; &  \frac{1}{A_k}\sum_{i=0}^k a_i \big(f(\vx_i) + \innp{\grad(\vx_i) - \vg_i, \vw - \vx_i}\big) + \frac{1}{A_k}\sum_{i=0}^k a_i \innp{\vg_i, \vw - \vx_i}\\
        &+ \frac{1}{2A_k}\|\vw - \vx_0\|_2^2 - \frac{1}{2A_k}\|\vw - \vx_0\|_2^2\\
        \geq &\; \frac{1}{A_k}\sum_{i=0}^k a_i \big(f(\vx_i) + \innp{\grad(\vx_i) - \vg_i, \vw - \vx_i}\big) + \frac{1}{A_k}\sum_{i=0}^k a_i \innp{\vg_i, \vv_k - \vx_i}\\
        &+ \frac{1}{2A_k}\|\vv_k - \vx_0\|_2^2 - \frac{1}{2A_k}\|\vw - \vx_0\|_2^2
        =: L_k,
    \end{align*}
    where in the last inequality we used the definition of $\vv_k := \proj_{\cx}(\vz_k)$ and $\vz_k = \vx_0 - \sum_{i=0}^k a_i \vg_i$ (by its definition). It remains to use the definition of $E_i^b$ and $G_k = f(\vy_k) - L_k.$ 
\end{proof}

We now formally prove that the error sequences defined in \eqref{eq:error-terms-def} satisfy the requirement that $E_0 \geq A_0 G_0 - \frac{1}{2}\|\vw - \vx_0\|_2^2$ and  $E_k \geq A_{k}G_k - A_{k-1} G_{k-1},$ which immediately implies \eqref{eq:gap-generic-bnd}. Obtaining a convergence bound for \eqref{eq:AGD+} then reduces to bounding the individual error terms in \eqref{eq:error-terms-def} using assumptions about $f$ and the gradient estimates $\vg_k.$

\begin{theorem}\label{thm:generic-AGD+-bound}
    Consider iterates of \eqref{eq:AGD+} for $k \geq 0$ and let $G_k$ be defined via \eqref{eq:G_k-def} and $E_k$ be defined via \eqref{eq:error-terms-def}. Then for any $\vw \in \cx,$ 
    \begin{equation}\notag
    \begin{aligned}
        A_0 G_0 &\leq \frac{1}{2}\|\vw - \vx_0\|_2^2 + E_0,\\
        A_{k}G_k - A_{k-1}G_{k-1} &\leq E_k, \;\text{ for }\; k \geq 1.
    \end{aligned}
    \end{equation}
    As a consequence, for all $k \geq 0,$
    \begin{equation}\notag
        f(\vy_k) - f(\vw) \leq \frac{(1/2)\|\vw - \vx_0\|_2^2 + \sum_{i=0}^k E_i}{A_k}.
    \end{equation}
\end{theorem}
\begin{proof}
    We start with bounding $A_0 G_0$, using its definition and $a_0 = A_0.$ We have
    \begin{align*}
        A_0 G_0 - \frac{1}{2}\|\vw - \vx_0\|_2^2 =\; & A_0(f(\vy_0) - f(\vx_0) - \innp{\vg_0, \vv_0 - \vx_0}) - \frac{1}{2}\|\vv_0 - \vx_0\|_2^2 + E_0^b\\
        =\; & A_0\Big(f(\vy_0) - f(\vx_0) - \innp{\grad(\vx_0), \vy_0 - \vx_0} - \frac{A_0}{2 {a_{0}}^2}\|\vy_0 - \vx_0\|_2^2\Big)\\
        &+ a_0 \innp{\grad(\vx_0) - \vg_0, \vv_0 - \vx_0}  + E_0^b\\
        =\; & E_0, 
    \end{align*}
    where we have used the definitions of error terms from \eqref{eq:error-terms-def}, $\vv_0 = \vy_0$, and $a_0 = A_0.$ 

    For the second inequality in the statement of the theorem, define 
    \[m_k(\vx) := \sum_{i=0}^k a_i \innp{\vg_i, \vx - \vx_i} + \frac{1}{2}\|\vx - \vx_0\|_2^2.\] 
    Then $A_k G_k - A_{k-1} G_{k-1}$ can be written as
    \begin{equation}\label{eq:E_k-bnd-1}
    \begin{aligned}
        A_k G_k - A_{k-1} G_{k-1} =\; & A_{k}f(\vy_k) - A_{k-1}f(\vy_{k-1}) - a_k f(\vx_k)\\
        &- m_k(\vv_k) + m_{k-1}(\vv_{k-1}) + E_k^b.
    \end{aligned} 
    \end{equation}
    Observe that $m_k$ is 1-strongly convex and minimized by $\vv_k,$ by the definition of $\vv_k = \proj_{\cx}(\vz_k).$ Thus, using its definition, we have
    \begin{align}
        m_k(\vv_k) - m_{k-1}(\vv_{k-1}) &= a_k \innp{\vg_k, \vv_k - \vx_k} + m_{k-1}(\vv_k) - m_{k-1}(\vv_{k-1}) \notag\\
        &\geq a_k \innp{\vg_k, \vv_k - \vx_k} + \frac{1}{2}\|\vv_k - \vv_{k-1}\|_2^2\notag\\
        &= a_k \innp{\grad(\vx_k), \vv_k - \vx_k} + \frac{1}{2}\|\vv_k - \vv_{k-1}\|_2^2 - E_k^v, \label{eq:E_k-bnd-2}
    \end{align}
    where the inequality holds because $m_{k-1}$ is 1-strongly convex and minimized by $\vv_{k-1}$ and the last line is by the definition of $E_k^v$. 
    On the other hand, by convexity of $f,$ we have
    \begin{equation}\label{eq:E_k-bnd-3}
        f(\vy_{k-1}) \geq f(\vx_k) + \innp{\grad(\vx_k), \vy_{k-1} - \vx_k}.
    \end{equation}
    Finally, plugging \eqref{eq:E_k-bnd-2} and \eqref{eq:E_k-bnd-3} back into \eqref{eq:E_k-bnd-1} and using that $\vy_k$ can be expressed as $\vy_k = \frac{A_{k-1}}{A_k} \vy_{k-1} + \frac{a_k}{A_k}\vv_k = \vx_k + \frac{a_k}{A_k}(\vv_k - \vv_{k-1})$ (from \eqref{eq:AGD+}), we get
    \begin{align*}
        &\; A_k G_k - A_{k-1}G_{k-1}\\
        \leq \; & E_k^b + E_k^v + A_k(f(\vy_k) - f(\vx_k) - \innp{\grad(\vx_k), \vy_k - \vx_k}) - \frac{1}{2}\|\vv_k - \vv_{k-1}\|_2^2 \\
        =\; & E_k^b + E_k^v + A_k\Big(f(\vy_k) - f(\vx_k) - \innp{\grad(\vx_k), \vy_k - \vx_k} - \frac{A_k}{{a_k}^2}\|\vy_k - \vx_{k}\|_2^2\Big)  \\
        =\; & E_k.
    \end{align*}
    The last inequality in the theorem statement follows immediately from $f(\vy_k) - f(\vw) \leq G_k$ (by Proposition~\ref{prop:G_k-construction}) and the first two inequalities in the theorem statement. 
\end{proof}

\subsubsection{Approximately Smooth Minimization}\label{sec:approx-smooth-opt}

We now show how to directly apply AGD+ to $f,$ with $\vg_k = \grad(\vx_k)$. In this case, clearly, $E_k^b = E_k^v = 0,$ $\forall k \geq 0,$ so to use the result from Theorem~\ref{thm:generic-AGD+-bound} to get concrete complexity bounds, we need to bound $E_k^s,$ for $k \geq 0.$ We do so using the \gradBMV\ property\markupadd{, based on the upper quadratic function approximation from Lemma \ref{lem:upper-quadratic-approx}}. \markupdelete{Our analysis in this case is based on interpreting the subgradient of $f$ as an inexact oracle for a smooth function. While this is a known idea in convex optimization \cite{nesterov2015universal,Devolder:2014}, we proceed differently from previous works when the points of interest lie further apart. Here we partition the line segment joining these two points and apply the \gradBMV\ property in each of these intervals. Aggregating these bounds then provides a sharper quadratic upper bound. }

\begin{lemma}\label{lemma:approx-smoothness-via-BMLV}
    Let $f: \rr^d \to \rr$ be \gradBMV\ at radius $r_0 > 0.$ For $r \in (0, r_0],$ let $\Llocr$ be the constant of maximum local variation of the subgradient of $f.$ Consider applying \eqref{eq:AGD+} to $f,$ initialized at an arbitrary $\vx_0 \in \cx.$ If for all $k \geq 0,$ $\frac{{a_k}^2}{A_k} \leq \frac{r}{\Llocr},$ then $E_k^s \leq \frac{{a_k}^2\Llocr^2}{2},$ $\forall k \geq 0.$
\end{lemma}
\begin{proof}
    Fix any $k \geq 0$ \markupadd{and recall that, by definition,
    \begin{equation}\notag
        E_k^s = A_k\Big(f(\vy_k) - f(\vx_k) - \innp{\grad(\vx_k), \vy_k - \vx_k}  - \frac{A_k}{2{a_k}^2}\|\vy_k - \vx_k\|_2^2\Big). 
    \end{equation}
    If $\|\vy_k - \vx_k\|_2 \leq r,$ then by Lemma \ref{lem:upper-quadratic-approx}, $f(\vy_k) - f(\vx_k) - \innp{\grad(\vx_k), \vy_k - \vx_k} \leq \Llocr\|\vy_k - \vx_k\|_2,$ and we can conclude that
    \begin{align*}
        E_k^s \leq \;& A_k\Big(\Llocr\|\vy_k - \vx_k\|_2  - \frac{A_k}{2{a_k}^2}\|\vy_k - \vx_k\|_2^2\Big)
        \leq \frac{{a_k}^2 \Llocr^2}{2},
    \end{align*}
    by Young's inequality. Thus the lemma claim holds in this case.
    }

    \markupadd{Now consider the remaining case that $\|\vy_k - \vx_k\|_2 > r.$ Then, by Lemma \ref{lem:upper-quadratic-approx}, $f(\vy_k) - f(\vx_k) - \innp{\grad(\vx_k), \vy_k - \vx_k} \leq \frac{\Llocr}{2r}\|\vy_k - \vx_k\|_2^2,$ and we conclude that
    \begin{align*}
         E_k^s \leq \;& A_k\Big(\frac{\Llocr}{2r}\|\vy_k - \vx_k\|_2^2  - \frac{A_k}{2{a_k}^2}\|\vy_k - \vx_k\|_2^2\Big) \leq 0,
    \end{align*}
    as $\frac{{a_k}^2}{A_k} \leq \frac{r}{\Llocr},$ by the lemma assumption. 
    }
    \markupdelete{Since $f$ is absolutely continuous, by the first theorem of calculus:
\begin{equation}\label{eq:yk-xk-FTC}
\begin{aligned}
    \; &f(\vy_k) - f(\vx_k) - \innp{\grad(\vx_k), \vy_k - \vx_k}\\
    =\; & \int_{0}^1 \innp{\grad(\vx_k + t(\vy_k - \vx_k)) - \grad(\vx_k), \vy_k - \vx_k} \dd t.
\end{aligned}
\end{equation}
If $\|\vy_k - \vx_k\|_2 \leq r,$ the integral on the right-hand side of \eqref{eq:yk-xk-FTC} is bounded by $\Llocr\|\vy_k - \vx_k\|_2$. Recalling the definition of $E_k^s$ from \eqref{eq:error-terms-def}, we have that in this case
\begin{align}\notag
    E_k^s = A_k\Big(\Llocr\|\vy_k - \vx_k\|_2 - \frac{A_k}{2{a_k}^2}\|\vy_k - \vx_k\|_2^2\Big) \leq  \frac{{a_k}^2\Llocr^2}{2}, 
\end{align}
as $\Llocr\|\vy_k - \vx_k\|_2 \leq \frac{{a_k}^2\Llocr^2}{2A_k} + \frac{A_k}{2{a_k}^2}\|\vy_k - \vx_k\|_2^2,$ by Young's inequality.} 
\markupdelete{
Now consider the case that $\|\vy_k - \vx_k\|_2 > r.$ Let $m = \lfloor \frac{\|\vy_k - \vx_k\|_2}{r} \rfloor + 1$. Then, we have
\begin{align*}
   \; & \int_{0}^1 \innp{\grad(\vx_k + t(\vy_k - \vx_k)) - \grad(\vx_k), \vy_k - \vx_k} \dd t\\
   =\; & \sum_{i = 1}^m \int_{(i-1)/m}^{i/ m} \innp{\grad(\vx_k + t(\vy_k - \vx_k)) - \grad(\vx_k), \vy_k - \vx_k} \dd t. 
\end{align*}
We now bound each of the integrals as follows:
\begin{align*}
   &\; \int_{(i-1)/m}^{i/ m} \innp{\grad(\vx_k + t(\vy_k - \vx_k)) - \grad(\vx_k), \vy_k - \vx_k} \dd t \\
   =\; & \int_{(i-1)/m}^{i/ m} \innp{\grad(\vx_k + t(\vy_k - \vx_k)) - \grad(\vx_k + (i-1)/m (\vy_k - \vx_k)), \vy_k - \vx_k} \dd t \\
   &+ \frac{1}{m}\sum_{j= 1}^{i-1} \innp{\grad(\vx_k + j/m (\vy_k - \vx_k)) - \grad(\vx_j + (j-1)/m (\vy_k - \vx_k)), \vy_k - \vx_k}\\
   &\; \leq \frac{1}{m}\Llocr \|\vy_k - \vx_k\|_2 + \frac{i-1}{m} \Llocr \|\vy_k - \vx_k\|_2 = \frac{i}{m} \Llocr\|\vy_k - \vx_k\|_2.
\end{align*}
Now summing over $i$ and plugging back into \eqref{eq:yk-xk-FTC}, we finally get
\begin{align*}
    f(\vy_k) - f(\vx_k) - \innp{\grad(\vx_k), \vy_k - \vx_k} &\leq \Llocr \|\vy_k - \vx_k\|_2 \sum_{i=1}^m \frac{i}{m}\\
    &\leq \Llocr \frac{m-1}{2}\|\vy_k - \vx_k\|_2 \\
    &\leq \frac{\Llocr}{2 r}\|\vy_k - \vx_k\|_2^2,
\end{align*}
where in the last inequality we used $m -1 = \lfloor \frac{\|\vy_k - \vx_k\|_2}{r} \rfloor \leq \frac{\|\vy_k - \vx_k\|_2}{r},$ by our choice of $m.$ Recalling the definition of $E_k^s$ from \eqref{eq:error-terms-def}, in this case clearly $E_k^s \leq 0$ whenever $\frac{{a_k}^2}{A_k} \leq \frac{r}{\Llocr}.$ }
\end{proof}

\begin{corollary}\label{cor:approx-smoothing-deterministic}
    Let $f: \rr^d \to \rr$ be a \gradBMV\ function at radius $r_0 > 0$ and assume $f$ is minimized by some $\vx_* \in \rr^d.$ For $r \in (0, r_0],$ let $\Llocr$ be the \gradBMV\ of $f.$ Consider applying \eqref{eq:AGD+} to $f,$ initialized at an arbitrary $\vx_0 \in \cx.$ If for all $k \geq 0,$ $\frac{{a_k}^2}{A_k} \leq \frac{r}{\Llocr},$ then $\forall k \geq 0,$
    \begin{equation}\notag
        f(\vy_k) - f(\vx_*) \leq \frac{\|\vx_* - \vx_0\|_2^2 + \sum_{i=0}^k {a_i}^2 \Llocr^2}{2A_k}.
    \end{equation}
    In particular, for any $\epsilon > 0,$ any $r \in (0, r_0]$ and the corresponding constant $\Llocr,$ there exists a choice of the step sizes $a_k \geq 0$ such that $f(\vy_k) - f(\vx_*) \leq \epsilon$ after 
    \begin{equation}\label{eq:approx-smoothing-complexity}
        k = O\bigg(\sqrt{\frac{\Llocr}{r \epsilon}}\|\vx_* - \vx_0\|_2 + \frac{\Llocr^2 \|\vx_* - \vx_0\|_2^2}{\epsilon^2}\bigg)
    \end{equation}
    iterations. Furthermore, these step sizes can be chosen adaptively w.r.t.\ $r$ and $\Llocr$, %
    with at most an additive logarithmic cost in the complexity.
\end{corollary}
\begin{proof}
    The first inequality in the statement  follows directly from Lemma~\ref{lemma:approx-smoothness-via-BMLV} and Theorem~\ref{thm:generic-AGD+-bound}. Observe that the condition $\frac{{a_k}^2}{A_k} \leq \frac{r}{\Llocr}$ for $k \geq 0$ corresponds to $a_0 = A_0 \leq \frac{r}{\Llocr}$ and $a_k \leq \frac{r}{\Llocr}\frac{1 + \sqrt{1 + 4A_{k-1}\Llocr/r }}{2}$ for $k \geq 1.$ By convention, let $A_{-1} = 0$ and define $a_k$ via
    \begin{equation}\label{eq:det-choice-ak}
        a_k = \min\Big\{\frac{\epsilon}{\Llocr^2} ,\, \frac{r}{\Llocr}\frac{1 + \sqrt{1 + 4A_{k-1}\Llocr/r }}{2}\Big\}. 
    \end{equation}
    Because $a_k \leq \frac{\epsilon}{\Llocr^2}$ for all $k \geq 0,$ we have that $\frac{\sum_{i=0}^k {a_i}^2 \Llocr^2}{2A_k} \leq \frac{\epsilon \sum_{i=0}^k {a_i}}{2 A_k} = \frac{\epsilon}{2}.$ Further, by the choice of $a_k$ from \eqref{eq:det-choice-ak}, we have that $A_k \geq \min\{(k+1) \frac{\epsilon}{\Llocr^2}, \frac{r}{\Llocr} (k+1)^2\}.$ As a consequence,
    \begin{align*}
        \frac{\|\vx_* - \vx_0\|_2^2}{2 A_k} &\leq \max\Big\{\frac{\|\vx_* - \vx_0\|_2^2}{2 (k+1) {\epsilon}/{\Llocr^2}},\, \frac{\|\vx_* - \vx_0\|_2^2}{2 ({r}/{\Llocr}) (k+1)^2}\Big\}.
    \end{align*}
    Thus, $\frac{\|\vx_* - \vx_0\|_2^2}{2 A_k} \leq \frac{\epsilon}{2}$ (which immediately leads to $f(\vy_k) - f(\vx_*) \leq \epsilon$) for $k \geq \frac{\Llocr^2 \|\vx_* - \vx_0\|_2^2}{\epsilon^2} + \sqrt{\frac{\Llocr}{r\epsilon}}\|\vx_* - \vx_0\|_2 - 1,$ completing the proof for the claimed number of iterations in \eqref{eq:approx-smoothing-complexity}.  

    Finally, to obtain the claimed bounds all that was needed was that $E_k^s \leq a_k \epsilon/2.$ Since $E_k^s$ is computable based on the iterates of the algorithm and step sizes set by the algorithm, this is a computable condition that can be checked. If the condition does not hold for the current choice of the step size $a_k$ in iteration $k \geq 0,$ the step size can be halved. Since we have already argued in \eqref{eq:det-choice-ak} a lower bound on the step size that suffices for the claimed iteration complexity, we get that the step size can be determined with at most a logarithmic cost using standard arguments based on the backtracking line search as in e.g., \cite{nesterov2015universal}. 
\end{proof}

A consequence of the above result is that we do not need to know ``the best'' radius $r$ a priori. The algorithm can automatically adapt to the ``best value'' of $r$ just assuming that the maximum local variation of the subgradient is bounded at any radius $r_0.$ Additionally, it is not hard to argue that the result stated in Corollary \ref{cor:approx-smoothing-deterministic} captures prior results on universal gradient methods under (global) weak smoothness or H\"{o}lder continuous gradient (see, e.g., \cite{nesterov2015universal,gasnikov2018universal}), where one assumes that there exist constants $\kappa \in [0, 1]$ and $M \in (0, \infty)$ so that \eqref{eq:weakly-smooth-def} holds. 
%
%
In particular, under \eqref{eq:weakly-smooth-def}, we have that the \gradBMV\ property applies for any $r > 0$ with $\Llocr = M r^\kappa$ and thus we can choose $r$ to minimize the oracle complexity from \eqref{eq:approx-smoothing-complexity}. Setting $r = \big(\frac{\epsilon^3}{M^3 D^2}\big)^{\frac{1}{1 + 3\kappa}}$ leads to the oracle complexity
\begin{equation}\label{eq:weakly-smooth-oracle-complexity}
    k = O\Big(\frac{M D^{1+\kappa}}{\epsilon}\Big)^{\frac{2}{1+3\kappa}},
\end{equation}
which is known to be optimal for this problem class \cite{Nemirovski:1985,Guzman:2015}. \markupadd{Further, if $f(\vx)$ can be decomposed into a sum of an $L$-smooth (gradient $L$-Lipschitz) and $M$-Lipschitz function, then $\Llocr \leq Lr + M$ for any $r > 0.$ In particular, taking $r = M/L$ and plugging into \eqref{eq:approx-smoothing-complexity}, we recover the optimal oracle complexity $O\big(\sqrt{\frac{L}{\epsilon}}\|\vx_* - \vx_0\|_2 + \frac{M^2\|\vx_* - \vx_0\|_2^2}{\epsilon^2}\big)$ obtained in \cite{grimmer2024optimal}.\footnote{\markupadd{We thank Ben Grimmer for this insight.}}}%

We see from this discussion that the provided result strictly generalizes known results for classical problem classes defined via weak smoothness. However, as noted before, \gradBMV\ class provides a more fine-grained characterization of complexity as it is possible for $\Llocr$ to be much smaller than the worst case $M r^\kappa$ for some $r$ and lead to a lower oracle complexity upper bound than stated in \eqref{eq:weakly-smooth-oracle-complexity}. Additionally, as argued earlier, $\Llocr$ can be finite (and small) even for functions that are neither globally Lipschitz nor (weakly) smooth.

\subsubsection{Randomized Smoothing}\label{sec:randomized-smoothing}

We now discuss how to obtain complexity bounds \markupadd{that potentially depend on the weaker \gradBMO\ property, using a randomized smoothing approach. We further show how the mean width of the subdifferential set around optima (see Lemma \ref{lemma:sparse-smooth-approx-err}) affects the oracle complexity, leading to the first positive result on parallelizing convex optimization for a nontrivial class of problems (e.g., piecewise-linear functions with polynomially many pieces).} \markupdelete{for \gradBMV\ and \gradBMO\ functions}. The idea is to apply AGD+ to $f_r$ defined by \eqref{eq:avg-fun-def}, using gradient estimates $\vg_k = \grad(\vx_k + r\vu_k).$ To do so, we need to show that $E_k^b$ and $E_k^v$ can be bounded in expectation, while $E_k^s$ will be at most zero under the appropriate step size choice, as $f_r$ is smooth (recall the results from Lemma~\ref{lem:smoothness_BMO}). 

\begin{proposition}\label{prop:bnded-err}
    Let $r > 0.$ Let $f:\rr^d \to \rr$ be a \gradBMV\ and \gradBMO\ function with constants $\widehat{L}_{2r}$ and $L_r$, respectively. Let $\vw \in \cx$ be arbitrary but fixed. Let $f_r$ be defined by \eqref{eq:avg-fun-def}. Let $\vx_k, \vz_k$ be the iterates of \eqref{eq:AGD+}. Then:
    \begin{equation}\notag
        \ee[E_k^b + E_k^v(\vw)] \leq {a_k}^2 L_r \widehat{L}_{2r}.%
    \end{equation}
\end{proposition}
\begin{proof}
    By the definition of $f_r,$ $\ee_{\vu_k \sim \unif(\cb)}[\grad(\vx_k + r \vu_k) ] = \nabla f_r(\vx_k);$ hence, we have $\ee_{\vu_k \sim \unif(\cb)}[\innp{\grad(\vx_k + r \vu_k) - \nabla f_r(\vx_k), \vv}]=0$ for any fixed  $\vv$, and, thus, $\ee_{\vu_k}[E_k^b] = 0$. 
    
    Let $\vv = \proj_{\cx}(\vz_{k-1} - a_k \markupadd{\nabla f_r(\vx_k)}\markupdelete{\grad(\vx_k)}).$ Recalling that $\vz_k = \vz_{k-1} - a_k \grad(\vx_k + r \vu_k),$ we have %
    \begin{align*}
       &\; \ee_{\vu_k \sim \unif(\cb)}[\innp{\grad(\vx_k + r \vu_k) - \nabla f_r(\vx_k), \vw - \proj_{\cx}(\vz_k)}] \\
        =&\; \ee_{\vu_k \sim \unif(\cb)}[\innp{\grad(\vx_k + r \vu_k) - \nabla f_r(\vx_k), \proj_{\cx}(\vz_{k-1} - a_k \markupadd{\nabla f_r(\vx_k)}\markupdelete{\grad(\vx_k)}) - \proj_{\cx}(\vz_{k})}]\\
        \stackrel{(i)}{\leq} &\; \ee_{\vu_k \sim \unif(\cb)}[\|\grad(\vx_k + r \vu_k) - \nabla f_r(\vx_k)\|_2 \|\proj_{\cx}(\vz_{k-1} - a_k \markupadd{\nabla f_r(\vx_k)}\markupdelete{\grad(\vx_k)}) - \proj_{\cx}(\vz_{k})\|_2]\\
        \stackrel{(ii)}{\leq} &\; a_k \ee_{\vu_k \sim \unif(\cb)}[\|\grad(\vx_k + r \vu_k) - \nabla f_r(\vx_k)\|_2\markupadd{^2} \markupdelete{\|\grad(\vx_k) - \grad(\vx_k + r \vu_k)\|_2}]\\
        \markupadd{\stackrel{(iii)}{\leq}} &\; \markupadd{a_k \widehat{L}_{2r} L_r,}
    \end{align*}
    where ($i$) is by Cauchy-Schwarz inequality, ($ii$) is by nonexpansiveness of the projection operator and $\vz_k = \vz_{k-1} - a_k \grad(\vx_k + r \vu_k)$\markupadd{, and ($iii$) is by Lemma \ref{lem:variance-bnd}}.\markupdelete{, ($iii$) is by the definition of $\Llocr,$ and ($iv$) is by the definition of $L_r.$} As a consequence, $\ee_{\vu_k}[E_k^v(\vw)] \leq {a_k}^2 \widehat{L}_{2r} L_r.$ 
    To complete the proof, it remains to take the expectation w.r.t.\ $\vu_0, \dots, \vu_{k-1}$ on both sides of the last inequality. 
\end{proof}

\begin{remark}\label{rem:rand-smoothing-variance}
\markupdelete{    A slight modification to the proof of Proposition \ref{prop:bnded-err}, where we instead take $\vv = \proj_{\cx}(\vz_{k-1} - a_k \nabla f_r(\vx_k))$ would, by the same argument, lead to the bound} \markupadd{Observe that in proving Proposition \ref{prop:bnded-err}, we showed that }
    \begin{equation}\label{eq:vr-like-bnd}
    \begin{aligned}
        &\; \ee_{\vu_k \sim \unif(\cb)}[\innp{\grad(\vx_k + r \vu_k) - \nabla f_r(\vx_k), \vw - \proj_{\cx}(\vz_k)}]\\ \leq \; & a_k \ee_{\vu_k \sim \unif(\cb)}[\|\grad(\vx_k + r \vu_k) - \nabla f_r(\vx_k)\|_2^2].
    \end{aligned}
    \end{equation}
    This quantity is clearly bounded by $a_k \widehat{L}_{2r}^2$ (which is generally a looser bound than what is provided in Proposition \ref{prop:bnded-err}; recall Example \ref{ex:Lr-vs-Llocr}),  but can also be bounded by defining a slightly stronger Lipschitz condition than \gradBMO, which is more ``variance-like:'' 
    \[\sup_{0 < \rho \leq r}\sup_{\vx \in \rr^d} \sqrt{\frac{1}{\vol(\cs_{\rho})}\int_{\cs_{\rho}}\|\grad(\vx + \vu) - \nabla f_r(\vx)\|_2^2\dd\vu} \leq \Tilde{L}_r < \infty.\] 
    The main usefulness of  \eqref{eq:vr-like-bnd} is that in parallel optimization settings this quantity can be reduced by taking more samples: by standard properties of the variance, the empirical average of $m$ samples $\grad(\vx_k + r\vu)$ for $\vu$'s drawn i.i.d.\ from $\unif(\cb)$ would reduce the right-hand side of \eqref{eq:vr-like-bnd} by a factor $\frac{1}{m}.$ Thus, for any $\epsilon > 0$, $m = \lceil\frac{\tilde{L}^2}{\epsilon}\rceil $ samples suffice to makes this quantity at most $\epsilon.$ \markupadd{We further note that without considerations related to reducing the variance via minibatching, it is possible to obtain a slightly tighter bound scaling with $ {a_k}^2 \Llocr \Lr$ instead of ${a_k}^2\widehat{L}_{2r}\Lr$ by choosing $\vv = \proj_{\cx}(\vz_{k-1} - a_k \grad(\vx_k))$ and following the same line of argument in the proof of Proposition \ref{prop:bnded-err}.}
\end{remark}

\paragraph{Choosing the smoothing radius $r$} Our standing assumption is that there exists a radius $r_0$ for which $\widehat{L}_{r_0}, L_{r_0}$ are bounded on the feasible set $\cx$. This is clearly true for Lipschitz continuous functions, but, as we have discussed before, can hold more generally. Observe that if this assumption holds for some $r_0,$ 
%
then for any $\epsilon_0 > 0,$ we can choose $r > 0$ such that $r \Llocr \leq \epsilon_0$ and $r L_r \leq \epsilon_0$, as for $r \in (0, r_0]$, we have $\Llocr \leq \widehat{L}_{r_0}$ and $L_r \leq L_{r_0}$.

Our ``ideal'' choice of a smoothing radius is the largest radius $r > 0$ such that $f_r(\vx_*) - f(\vx_*) \leq \epsilon/2,$ where $\epsilon > 0$ is the target error and $\vx_*$ a minimizer of $f$. Using Lemma~\ref{lemma:generic-smooth-approx-err}, to have $f_r(\vx_*) - f(\vx_*) \leq \epsilon/2,$ it suffices that $r L_r \leq \epsilon/2.$ Alternatively, based on Lemma~\ref{lemma:sparse-smooth-approx-err} an Remark~\ref{rem:bounded-Gaussian-width}, if the Goldstein $r$-subdifferential at $\vx_*$, $\partial_r f(\vx_*)$, is contained in a polytope $K_{\vx_*}$ of Euclidean diameter $D_{\vx_*}$, 
    then %
    $
        f_r(\vx_*) - f(\vx_*) = O\big(r D_{\vx_*}  \sqrt{\frac{\ln|{ \ext}(K_{\vx_*})|}{d}}\big).
    $ 
Observe that it suffices that such a condition holds only for the Goldstein $r$-subdifferential at $\vx_*$. %
The reason for considering this condition is that it allows choosing a potentially much larger smoothing radius $r.$ In particular, due to the \gradBMV\ property, it is possible to choose an enclosing polytope $K_{\vx_*}$ to be of diameter $C \Llocr,$ for any $C > 1.$ Here, the tradeoff in choosing $C$ is that we want the polytope to have as few vertices as possible while keeping $C$ as an absolute constant. In particular, if $K_{\vx_*}$ has $\mathrm{poly}(d)$ vertices, then we can ensure $f_r(\vx_*) - f(\vx_*) \leq \epsilon/2$ with $r \Llocr = O\big(\epsilon \sqrt{\frac{d}{\ln(d)}}\big)$. We summarize the resulting complexity bounds as follows. 

\begin{corollary}\label{cor:cvx-rand-smoothing}
    Let $f: \rr^d \to \rr$ be a convex function minimized by some $\vx_* \in \cx$ on a closed convex set $\cx.$ Suppose that there exists a radius $r_0 > 0$ such that $f$ is a \gradBMV\ function for $r_0$. Given a target error $\epsilon > 0$, suppose the radius $r > 0$ is chosen so that $f_r(\vx_*) - f(\vx_*) \leq \epsilon/2$, where $f_r$ is defined by \eqref{eq:avg-fun-def}. Let $\lambda_r$ denote the \markupdelete{smoothness parameter of $f_r.$}\markupadd{Lipschitz constant of $\nabla f_r.$}  Consider applying \eqref{eq:AGD+} to $f_r(\vx)$, using stochastic gradient oracle $\grad(\vx + r\vu),$ $\vu \sim \unif(\cb),$ for an arbitrary initial point $\vx_0 \in \cx$ and $\frac{{a_k}^2}{A_k} \leq \frac{1}{\lambda_r},$ $\forall k \geq 0.$ Then, $\forall k \geq 0,$
    \begin{equation}\notag
        \ee[f(\vy_k) - f(\vx_*)] \leq \frac{\epsilon}{2} + \frac{\frac{1}{2}\|\vx_* - \vx_0\|_2^2 + \sum_{i=0}^k {a_i}^2 \Llocr L_r}{A_k}.
    \end{equation}
    In particular, there exist step sizes $\{a_i\}_{i=0}^k$ such that $\ee[f(\vy_k) - f(\vx_*)] \leq \epsilon$ after at most 
    \begin{equation}\label{eq:rand-smoothing-complexity}
        k = O\bigg(\sqrt{\frac{\lambda_r}{\epsilon}}\|\vx_* - \vx_0\|_2 + \frac{\Llocr L_r \|\vx_* - \vx_0\|_2^2}{\epsilon^2}\bigg)
    \end{equation}
    iterations. In the above bound, $\lambda_r$ satisfies the following
    \begin{equation}\notag
        \begin{aligned}
            \lambda_r = O\bigg(\min\bigg\{ \frac{{L_r}^2 d}{\epsilon},\, \frac{L_r \Llocr \sqrt{d}}{\epsilon},\, 
            \frac{{\Llocr}^2 \sqrt{\ln(|\ext(K_{\vx_*})|)}}{\epsilon}\bigg\}\bigg).
        \end{aligned}
    \end{equation}
\end{corollary}

Before proving the corollary, a few remarks are in order. Observe first that the second term in \eqref{eq:rand-smoothing-complexity} can be replaced by a term absorbed by the first one by taking multiple samples $\grad(\vx_k + r\vu_k)$ in parallel and choosing $\vg_k$ as their average (see Remark~\ref{rem:rand-smoothing-variance}). This means that it is possible to parallelize the method using $\mathrm{poly}(1/\epsilon, \Llocr, \|\vx_* - \vx_0\|_2)$ oracle queries per round and have the first term in \eqref{eq:rand-smoothing-complexity} determine the number of parallel rounds. See, e.g., \cite{bubeck2019complexity,duchi2012randomized} for similar ideas used in nonsmooth Lipschitz continuous optimization. 

Second, similar to the result from Corollary~\ref{cor:approx-smoothing-deterministic}, the value of the Lipschitz constant of $f$ plays no role in the oracle complexity bound in Corollary \ref{cor:cvx-rand-smoothing}. It is possible that a function is not Lipschitz continuous at all (recall the examples from Section~\ref{sec:bounded-local-var}), yet we get complexity bounds that are similar to the complexity of nonsmooth Lipschitz convex optimization, at least in some regimes of the problem parameters.

Because $L_r$ can generally be much smaller than $\Llocr$, it is not clear a priori which term in the minimum determines the value of $\lambda_r$ (and thus the oracle complexity in \eqref{eq:rand-smoothing-complexity}). When $L_r = \Omega(\Llocr/\sqrt{d}),$ then we have $\lambda_r = O(\min\{\frac{L_r \Llocr \sqrt{d}}{\epsilon},\, \frac{{\Llocr}^2 \sqrt{\ln(|\ext(K_{\vx_*})|)}}{\epsilon}\})$. In particular, when the Goldstein subdifferential $\partial_r f(\vx_*)$ is contained in a polytope of diameter $O(\Llocr)$ with $\mathrm{poly}(d)$ vertices, $\lambda_r$ is nearly independent of the dimension (the dependence on the dimension becomes $\sqrt{\ln(d)}$). 

A surprising aspect of this result is that not only do we get complexity that depends on $\Llocr$, which can be much smaller than the objective's Lipschitz constant, but in this case it is also possible to obtain a parallel algorithm that makes $\mathrm{poly}(\Llocr \|\vx_* - \vx_0\|_2/\epsilon)$ queries per round and has depth (number of parallel rounds) that scales with $O(\frac{\Llocr (\ln(d))^{1/4}\|\vx_* - \vx_0\|_2}{\epsilon}).$ As a consequence, we get  \emph{the first example of a class of nonsmooth optimization problems for which parallelization leads to improved depth of the algorithm that is essentially dimension-independent}. As a specific example, nonsmooth $M$-Lipschitz-continuous functions that can be expressed as or closely approximated by a maximum of polynomially many in $d$ and $1/\epsilon$ linear functions have parallel complexity at most $O(\frac{M (\ln(d/\epsilon))^{1/4}\|\vx_* - \vx_0\|_2}{\epsilon})$ -- significantly lower than the sequential complexity $O(\frac{M^2 \|\vx_* - \vx_0\|_2^2}{\epsilon^2})$ for any $d$ that is polynomial in $1/\epsilon.$ 

This last statement seems at odds with parallel oracle complexity lower bounds for standard Euclidean settings \cite{diakonikolas2020lower, bubeck2019complexity,woodworth2018graph,balkanski2019exponential}, which are all based on a max-of-linear hard probabilistic instance with $\mathrm{poly}(d, 1/\epsilon)$ components originally introduced by Nemirovski \cite{nemirovski1994parallel}. The apparent contradiction is resolved by observing that all these existing lower bounds become informative for $d \gg 1/\epsilon^2$ and crucially rely on the \emph{informative} queries being confined to the unit Euclidean ball. By contrast, our randomized smoothing approach in this case relies on queries to $\nabla f(\vx + r\vu),$ $\vu \sim \unif(\cb)$, with $r \approx \sqrt{d}\epsilon$, meaning that all queries fall well \emph{outside} the unit ball with high probability and thus the existing lower bounds do not apply. %

\begin{proof}[Proof of Corollary~\ref{cor:cvx-rand-smoothing}]
    First, because $f_r(\vx_*) - f(\vx_*) \leq \epsilon/2$ and $f_r(\vy_k) \geq f(\vy_k)$ (due to convexity of $f,$ by Lemma~\ref{lemma:generic-smooth-approx-err}), we have 
    \begin{equation}\label{eq:f-to-fr-rand-smoothing}
        \ee[f(\vy_k) - f(\vx_*)] \leq \epsilon/2 + \ee[f_r(\vy_k) - f_r(\vx_*)],
    \end{equation}
    so we only need to focus on bounding $\ee[f_r(\vy_k) - f_r(\vx_*)]$, which we do using Theorem~\ref{thm:generic-AGD+-bound} and Proposition~\ref{prop:bnded-err}. In particular, because $f_r$ is $\lambda_r$-smooth, we have that
    \begin{equation}\notag
        E_k^s \leq A_k\Big(\frac{\lambda_r}{2} - \frac{A_k}{2 {a_{k}}^2}\Big)\|\vy_k - \vx_k\|_2^2 \leq 0,
    \end{equation}
    as $\frac{{a_k}^2}{A_k} \leq \frac{1}{\lambda_r},$ by assumption. Thus, applying Theorem~\ref{thm:generic-AGD+-bound} and Proposition~\ref{prop:bnded-err}, we have
    \begin{equation}\label{eq:rand-smoothing-bnd}
        \ee[f_r(\vy_k) - f_r(\vx_*)] \leq \frac{\frac{1}{2}\|\vx_* - \vx_0\|_2^2 + \sum_{i=0}^k {a_i}^2 L_r \Llocr}{A_k},
    \end{equation}
    which, combined with \eqref{eq:f-to-fr-rand-smoothing} leads to the first inequality in Corollary \ref{cor:cvx-rand-smoothing}. 

    The bound on $\lambda_r$ follows from Lemma~\ref{lem:smoothness_BMO} and upper bounds on $f_r(\vx_*) - f(\vx_*)$ in Lemmas \ref{lemma:generic-smooth-approx-err} and Lemma~\ref{lemma:sparse-smooth-approx-err}, by setting those upper bounds to $\epsilon/2$ and solving for $r$. 

    Finally, it remains to argue that there is a choice of step sizes $a_i$ such that $f_r(\vy_k) - f_r(\vx_*) \leq \epsilon/2$ in the number of iterations stated in \eqref{eq:rand-smoothing-complexity}. This is done using similar ideas as in \cite{ghadimi2012optimal}. In particular, define $a_i$'s via $\frac{{a_i}^2}{A_i} = \beta$ (this enforces $a_0 = \beta$ and for $i \geq 1$ is a quadratic equality with a unique solution, using $A_i = A_{i-1} + a_i$) for $\beta > 0$ to be specified shortly. It is well-known that in this case for $i \geq 1,$ $a_i = \Theta(\beta i)$ and $A_i = \Theta(\beta i^2).$ Thus the bound on $f_r(\vy_k) - f_r(\vx_*)$ becomes
    \begin{equation}\label{eq:choosing-gamma}
        \ee[f_r(\vy_k) - f_r(\vx_*)] = O\Big(\frac{\|\vx_* - \vx_0\|_2^2}{\beta k^2} + k \beta L_r \Llocr\Big).
    \end{equation}
    In particular, $\beta = \frac{\|\vx_* - \vx_0\|_2}{\sqrt{L_r\Llocr}k^{3/2}}$ balances the terms on the right-hand side of \eqref{eq:choosing-gamma}, but we also need $\beta \leq \frac{1}{\lambda_r}$ to satisfy the assumption that $\frac{{a_i}^2}{A_i} \leq \frac{1}{\lambda_r}.$ Hence, we choose $\beta = \min\{\frac{1}{\lambda_r}, \frac{\|\vx_* - \vx_0\|_2}{\sqrt{L_r\Llocr}k^{3/2}}\}$. Since $\beta \leq \frac{\|\vx_* - \vx_0\|_2}{\sqrt{L_r\Llocr}k^{3/2}},$ it follows that $k \beta L_r \Llocr \leq \epsilon/4$ for $k = O(\frac{L_r \Llocr \|\vx_* - \vx_0\|_2^2}{\epsilon^2}).$ On the other hand, by the choice of $\beta,$ we have $\frac{\|\vx_* - \vx_0\|_2^2}{\beta k^2} \leq \epsilon/4$ for $k = O(\max\{\frac{L_r \Llocr \|\vx_* - \vx_0\|_2^2}{\epsilon^2},\, \sqrt{\frac{\lambda_r}{\epsilon}}\|\vx_* - \vx_0\|_2\}),$ hence the claimed bound \eqref{eq:rand-smoothing-complexity} follows. 
\end{proof}

\subsection{Goldstein's Method and Nonconvex Optimization}\label{sec:nonconvex-Goldstein}

Interestingly, our framework also proves to be useful in the nonconvex setting. In particular, in this section we provide refined complexity results for the Goldstein method for approximating stationary points in locally Lipschitz (nonconvex) optimization. To do so, we adapt the results from \cite{davis2022gradient} -- which pertain to Lipschitz objectives -- to the \gradBMV\ class of functions studied in this paper.

In what follows, we consider a function $f$ with (local) Lipschitz constant $M$ and local variation of subgradients bounded by $\widehat L_r$ for some $r>0$. For any vector $\vg$, we let $\hat\vg:=\frac{\vg}{\|\vg\|_2}$. The following lemma can be seen as an extension of \cite[Lemma 2.2]{davis2022gradient} to this setting. 
In what follows, we will make the particular choice $r=2\delta$. 

\begin{lemma} \label{lem:norm_contraction_Goldstein}
Let $f:\rr^d\to\rr$ be an $M$-locally Lipschitz function. Let $\vg\in \partial_{\delta} f(\vx)$ be such that $\|\vg\|_2>\epsilon$ and
\begin{equation} \label{eqn:non_suff_decrease} f(\vx-\delta\hat \vg)-f(\vx)\geq -\frac{\delta}{2}\|\vg\|_2. 
\end{equation}
Let $p\geq 1$ be an integer, 
$\vh\sim \unif(\cb_{2^{-p}}(\vg))$, and $\vu=\grad(\vy)$, where $\vy\sim \unif([\vx,\vx-\delta\hat \vh])$. Then 
\[ \mathbb{E}\langle \vu,\vg\rangle \leq 
\frac12\|\vg\|_2^2+2M\|\vg\|_2 2^{-p} 
\]
In particular, if $p\geq \log_2\big(\frac{12M}{\epsilon}\big)$, there exists $\lambda\in[0,1]$ such that if $\vz=\vg+\lambda(\vu-\vg)$, then
\[ 
\mathbb{E}\|\vz\|_2^2 
\leq 
\left\{
\begin{array}{rl}
\|\vg\|_2^2-\frac{\|\vg\|_2^4}{9\Llocr^2},& \mbox{ if } \|\vg\|_2^2\leq 3\Llocr\\
\frac{2}{3}\|\vg\|_2^2,  &\mbox{ if }\|\vg\|_2^2 > 3\Llocr
\end{array}
\right.
,\]
where $r = 2\delta.$
\end{lemma}

\begin{proof}
Since $\hat \vh$ is generic, $f$ is differentiable in almost every point on the interval $[\vx,\vx-\delta\hat \vh]$. Hence, by inequality \eqref{eqn:non_suff_decrease} and the first theorem of calculus,
\begin{align*}
\frac12\|\vg\|_2 &\geq \frac{f(\vx)-f(\vx-\delta\hat\vg)}{\delta}
=\frac{f(\vx)-f(\vx-\delta\hat\vh)}{\delta}+\frac{f(\vx-\delta\hat\vh)-f(\vx-\delta\hat\vg)}{\delta}\\
&\geq \frac{1}{\delta}\int_0^{\delta}\langle \grad(\vx-\tau\hat \vh),\hat\vh\rangle d\tau- M\|\hat\vh-\hat\vg\|\\ &\geq \mathbb{E}\langle \vu,\hat\vg\rangle-2M\|\hat \vh-\hat \vg\|_2.
\end{align*}
In particular,
\[ 
\mathbb{E}\langle \vu,\vg\rangle
\leq \frac12\|\vg\|_2^2+2M\|\vg\|_2 2^{-p}.
\]

Notice that if $p\geq \log_2\big(\frac{12M}{\epsilon}\big)$, then 
$\mathbb{E}\langle \vu,\vg\rangle \leq \frac12\|\vg\|_2^2+2M\|\vg\|_2 2^{-p}\leq \frac{2}{3}\|\vg\|_2^2.$
For the rest of the proof we impose this assumption. 
Now, we consider the random variable $\vz=\vg+\lambda(\vu-\vg)$. We have 
\begin{align*}
\mathbb{E}\|\vz\|_2^2 &= \|\vg\|_2^2+2\lambda \mathbb{E}\langle\vg,\vu-\vg\rangle+\lambda^2\mathbb{E}\|\vu-\vg\|_2^2 
 \leq \big(1-\frac{2\lambda}{3}\big)\|\vg\|_2^2+\lambda^2\Llocr^2.
\end{align*}
Here we have two choices. First, if $\|\vg\|_2^2<3\Llocr^2$, then we can set $\lambda=\|\vg\|_2^2/[3\Llocr^2]$, leading to
$\mathbb{E}\|\vz\|^2\leq \|\vg\|_2^2\big(1-\frac{\|\vg\|_2^2}{9\Llocr^2}\big).$ 
Otherwise, set $\lambda=1$, which leads to
$\mathbb{E}\|\vz\|^2\leq \frac13\|\vg\|_2^2+\Llocr^2\leq \frac23\|\vg\|_2^2,$ 
completing the proof.
\end{proof}

With this technical lemma, the algorithm and its analysis follow naturally. The algorithm performs gradient descent-style steps using a vector from the Goldstein subdifferential chosen at random. If the subgradient has norm smaller than $\epsilon$, then the algorithm stops and outputs the current iterate; alternatively, if the subgradient provides sufficient decrease, we update the vector taking a normalized step of length $\delta$ in this direction; finally, if neither of the above holds, the algorithm enters a loop where -- due to Lemma \ref{lem:norm_contraction_Goldstein} -- we can find elements in the Goldstein subdifferential which decrease the subgradient norm multiplicatively. In particular, either this loop leads to a ``sufficient decrease'' step, or we obtain a subgradient with norm less than $\epsilon$. The convergence analysis of the algorithm follows from a combination of the sufficient decrease steps and a bound on the length of each internal loop.

\begin{algorithm}
\caption{Interpolated Normalized Gradient Descent}\label{alg:INGD}
\begin{algorithmic}[1]
\STATE {{\bf Initialization}: $x_0\in\rr^d$}
\FOR{$k= 0,\ldots,K$}
    \STATE{$\vu\sim \unif(\cB_{\delta}(\vx_k))$}
    \STATE{$\vg_k=\grad(\vu)$} 
    \WHILE{\texttt{True}}
        \IF{$\|\vg_k\|_2 \leq \epsilon$}
            \STATE{\bf{Stop algorithm and return $\vx_k$}
        }
        \ELSE
        \IF{$f(\vx_k-\delta\hat\vg_k)-f(\vx_k)\leq -\frac{\delta}{2}\|\vg_k\|_2$}
           \STATE{ $\vx_{k+1} \gets \vx_k - \delta\hat \vg_k$}
            \STATE {\bf Break while}
        \ELSE
            \WHILE{\texttt{True}}
               \STATE{ $\vh\sim \unif(\cB_{2^{-p}}(\vg_k))$ with $p=\log_2(12M/\epsilon)$}
              \STATE{  $\vy\sim\unif([\vx_k,\vx_k-\delta\hat \vh])$}
               \STATE{ $\vu \gets \grad(\vy)$}
               \STATE{ $\lambda \gets \min\big\{1,\frac{\|\vg_k\|_2^2}{3\Llocr^2}\big\}$}
               \STATE{ $\tilde\vg\gets \vg_k+\lambda(\vu-\vg_k)$}
                \IF{$\|\tilde\vg\|_2^2\leq \|\vg_k\|_2^2-\frac{\|\vg_k\|_2^4}{18L_{2\delta}^2}$ or $\|\tilde\vg\|_2^2\leq \frac{3\|\vg_k\|_2^2}{4}$}
                   \STATE{ $\vg_k\gets\tilde\vg$}
                    \STATE {\bf Break while}
                \ENDIF    %
            \ENDWHILE
        \ENDIF
        \ENDIF
    \ENDWHILE
\ENDFOR
\RETURN {$\vx_K$}
\end{algorithmic}
\end{algorithm}

\begin{theorem} \label{thm:convergence_INGD}
Let $f:\rr^d\to\rr$ be an $M$-locally Lipschitz function with \gradBMV\ constant $\Llocr$ for $r>0$. Let $\vx_0\in\rr^d$ be such that $f(\vx_0)- f_*\leq \Delta$. Then, with probability at least $1-\beta$,  
Algorithm \ref{alg:INGD} outputs a $(\delta,\epsilon)$-stationary point after $O\Big( \frac{\Delta\Llocr^2}{\epsilon^3\delta}\ln\big(\frac{\Delta}{\epsilon\delta\beta}\big)\Big)$ (sub)gradient oracle queries to $f$.
\end{theorem}

\begin{proof}
Due to our assumption on the suboptimality of $\vx_0$, we note that the condition $f(\vx_k-\delta\hat\vg_k)-f(\vx_k)\leq -\frac{\delta}{2}\|\vg_k\|_2$ can be satisfied only at most $K=\frac{2\Delta}{\epsilon\delta}$ times before stopping. Indeed, if there are $K$ descent steps, 
\[ \frac{K\epsilon\delta}{2}\leq \sum_{k=0}^{K-1} [f(\vx_{k+1})-f(\vx_{k})] = f(\vx_0)-f(\vx_K) \leq \Delta. \]

On the other hand, the length of the innermost while loop can be bounded using Lemma \ref{lem:norm_contraction_Goldstein}. We proceed by considering the possible cases for the value of $\|\vg_k\|_2$ in iteration $k$. First,  if $\|\vg_k\|_2^2\leq 3\Llocr$, then by the chosen values of $p$ and $\lambda$, we have that $\mathbb{E}\|\tilde\vg\|_2^2\leq  \|\vg_k\|_2^2\Big(1-\frac{\|\vg_k\|_2^2}{9\Llocr^2}\Big)$. Now, by Markov's inequality,
\[ 
\mathbb{P}\Big[ \|\tilde\vg\|^2>  \|\vg_k\|_2^2\Big(1-\frac{\|\vg_k\|_2^2}{18\Llocr^2}\Big) \Big]
\leq \frac{1-\frac{\|\vg_k\|_2^2}{9\Llocr^2}}{1-\frac{\|\vg_k\|_2^2}{18\Llocr^2}}
\leq 1 - \frac{\|\vg_k\|_2^2/[18\Llocr^2]}{1-\frac{\|\vg_k\|_2^2}{18\Llocr^2}}
\leq 1-\frac{\|\vg_k\|_2^2}{15 \Llocr^2}.
\]
This implies that after $J_k$ passes over the inner loop, the probability of exiting the loop is $\big(1-\frac{\|\vg_k\|_2^2}{15 \Llocr^2}\big)^{J_k}$, and to make this probability smaller than $\beta/K$, it suffices to have  $J_k =\frac{15\Llocr^2}{\epsilon^2}\ln\big(\frac{K}{\beta}\big)$. 
In the remaining case $\|\vg_k\|_2^2>3\Llocr^2$, we have by a similar 
reasoning: 
$
\mathbb{P}\big[ \|\tilde\vg\|^2>  \frac34\|\vg_k\|_2^2 \big]
\leq \frac89.
$ 
Hence, the length of the inner loop is at most $J_k=[\ln(8/9)]^{-1} \ln\big(\frac{K}{\beta}\big).$

In conclusion, by the union bound and the previous reasoning,  with probability at most $1-\beta$, the number of iterations (subgradient oracle queries) that the algorithm makes is at most $\sum_{k=0}^K J_k = O\Big( \frac{\Delta\Llocr^2}{\epsilon^3\delta}\ln\big(\frac{\Delta}{\epsilon\delta\beta}\big)\Big)$.
\end{proof}

A few remarks are in order here. First, same as in the settings of convex optimization considered earlier in this section, the resulting oracle complexity upper bound is \emph{independent} of the (local) Lipschitz constant of $f.$ However, here we crucially rely on the assumption that the  (local) Lipschitz constant $M$ of $f$ is finite to ensure that vectors $\vh$ utilized by the algorithm are random. We note that the constant $M$ need not be known to the algorithm; instead, it can be adaptively estimated with only a logarithmic overhead in the complexity, by simply choosing $p$ growing within the innermost while loop. On the other hand, making the algorithm independent of $\Llocr$ (and thus fully parameter-free) appears to be more challenging and is an interesting question for future research. We note that obtaining a parameter-free version of the Goldstein method with provable convergence guarantees is open even in the case of Lipschitz nonsmooth nonconvex optimization studied in \cite{davis2022gradient,zhang2020complexity}.

Finally, because $\Llocr \leq 2L,$ Theorem \ref{thm:convergence_INGD} recovers the previously known bounds for Lipschitz-continuous nonconvex nonsmooth optimization \cite{davis2022gradient,zhang2020complexity}. By a similar reasoning as in Section \ref{sec:approx-smooth-opt}, we can also draw conclusions about convergence of Algorithm \ref{alg:INGD} in $(M, \kappa)$-weakly smooth settings, for $\kappa \in (0, 1].$ In this case, the function is differentiable and its gradient is H\"{o}lder-continuous. From the definition of {$\partial_r f(\vx)$}  we can further deduce that the output point $\vx_K$ of Algorithm \ref{alg:INGD} in this case satisfies, by the triangle inequality:
$
    \|\nabla f(\vx_K)\|_2 \leq \epsilon + M r^\kappa, 
$
where we recall $r = 2\delta.$ In particular, if $r = (\frac{\epsilon}{M})^{1/\kappa},$ we have  $\|\nabla f(\vx_K)\|_2 \leq 2 \epsilon.$ The total number of oracle queries in this case is $k = \tilde{O}\big(\frac{\Delta M^2 r^{2\kappa}}{\epsilon^3 r}\big) = \tilde{O}\big(\frac{\Delta M^{1/\kappa}}{\epsilon^{1 + 1/\kappa}}\big).$ For the case of smooth functions ($\kappa = 1$), this oracle complexity is optimal up to a logarithmic factor \cite{carmon2017lower-II}, and the same result was established for Goldstein's method in \cite{zhang2020complexity}, using a different argument based on a descent condition being satisfied in each iteration. Here we obtain oracle complexity results for all weakly smooth functions and for nonsmooth Lipschitz functions,  based on \emph{one} result, stated in Theorem~\ref{thm:convergence_INGD}. It is an open question if this oracle complexity upper bound is (near) optimal for $\kappa \in (0, 1),$ though we conjecture it is.  

\section{Conclusion}

We introduced new classes of nonsmooth optimization problems based on local (maximum or average) variation of the function's subgradient and showed that this perspective generalizes  classical results in optimization based on Lipschitz continuity and weak smoothness, leading to more fine-grained oracle complexity bounds. On a conceptual level, one bottom line of our work is that it is not the growth of the function that  determines complexity, but how its slope  changes over small regions. Another is that complexity of parallel convex optimization depends on the complexity of the subdifferential set around optima. %

As a byproduct of our results, we showed that -- contrary to prior belief based on lower bounds \cite{nemirovski1994parallel,diakonikolas2020lower, balkanski2019exponential,woodworth2018graph,bubeck2019complexity} -- the complexity of parallel optimization can, in fact, be improved even in high-dimensional settings under fairly mild assumptions about the complexity of the subdifferential set around optima. All that is needed is that the algorithm is given slightly more power: to be able to query points outside the unit ball. As a specific example, functions that can be expressed as or closely approximated by piecewise linear functions with polynomially many pieces in the dimension $d$ and the inverse accuracy $1/\epsilon$ can benefit from parallelization in terms of sequential oracle complexity (parallel ``depth'') by a factor $\tilde{\Omega}(1/\epsilon)$ so long as we are allowed to query them at points at distance   $O(\epsilon\sqrt{d})$ from the feasible set. {Despite the seemingly specific nature of this example, minimizing a maximum of linear functions has been a key focus of research in nonsmooth convex optimization, and some of the most important developments in this area were inspired by this example \cite{Nesterov:2005minimizing,nemirovski2004prox}. Our results not only provide an alternative view of this setting, but also a broader perspective on functional classes that are amenable to parallelization by randomization.}

Some interesting questions that merit further investigation remain. For example, can recent techniques on adaptive, parameter-free optimization \cite{malitsky2020adaptive,li2023simple} be generalized to our \gradBMV\ class and lead to algorithms that are both universal and parameter free but avoid the line search used in our result? Can parallel optimization methods based on randomized smoothing be made parameter-free while maintaining the oracle complexity benefits described in our work? Is it possible for the Goldstein's method we analyzed on the \gradBMV\ class in Section~\ref{sec:nonconvex-Goldstein} to be made completely parameter-free?  Finally, while we did not pursue this direction, there are more sophisticated algorithms for parallel convex optimization based on randomized smoothing and higher-order optimization \cite{bubeck2019complexity,chakrabarty2024parallel}. It seems plausible that the use of such techniques could further reduce parallel complexity of nonsmooth optimization for piecewise linear functions. It would be interesting to formally establish such a result. 

\section*{Acknowledgements}

J.\ Diakonikolas's research was partially supported by the Air Force Office of Scientific Research under award number FA9550-24-1-0076 and by the U.S.\ Office of Naval Research under contract number  N00014-22-1-2348. Any opinions, findings and conclusions or recommendations expressed in this material are those of the author(s) and do not necessarily reflect the views of the U.S. Department of Defense. 

C.~Guzm\'an's research was
partially supported by
INRIA Associate Teams project, ANID FONDECYT 1210362 grant, ANID Anillo ACT210005 grant, and National Center for Artificial Intelligence CENIA FB210017, Basal ANID.

\bibliographystyle{abbrv}
\bibliography{references.bib}

\end{document}

%% file: prologue.tex
\usepackage{inputenc}
\usepackage{amsmath, dsfont}
\usepackage{times}
\usepackage{amssymb}
\usepackage{amsthm}
\usepackage{bm}
\usepackage{graphicx}
\usepackage{caption}
 \usepackage[margin = 1in]{geometry}
 \usepackage{cite}
\usepackage{color}
\usepackage{subfigure}
\usepackage{comment}
\usepackage{enumitem}
\usepackage{mathtools}
\usepackage{thmtools, thm-restate}
\usepackage[ruled]{algorithm2e}
\usepackage{verbatim}
\usepackage[noend]{algorithmic}

\SetKwComment{Comment}{/* }{ */}

\usepackage[colorlinks,urlcolor=blue,citecolor=blue,linkcolor=blue]{hyperref}

\DeclareMathOperator{\argmin}{\mathrm{argmin}}

\usepackage[colorinlistoftodos]{todonotes}

\newcommand{\gradBMV}{BVG$_{\mbox{\footnotesize max}}$}
\newcommand{\gradBMO}{BVG$_{\mbox{\footnotesize avg}}$}
\newcommand{\grad}{\gamma_f}

\newcommand{\ext}{\operatorname{vert}}
\newcommand{\vol}{\operatorname{vol}}
\newcommand{\unif}{\operatorname{Unif}}
\newcommand{\innp}[1]{\left\langle #1 \right\rangle}

\newcommand{\mA}{\mathbf{A}}
\newcommand{\ones}{\mathds{1}}
\newcommand{\zeros}{\textbf{0}}
\newcommand{\vx}{\mathbf{x}}
\newcommand{\va}{\mathbf{a}}
\newcommand{\vb}{\mathbf{b}}
\newcommand{\Llocr}{\widehat{L}_r}

\newcommand{\Lr}{{L}_r}

\newcommand{\dd}{\mathrm{d}}
\newcommand{\cx}{\mathcal{X}}

\newcommand{\cb}{\mathcal{B}}
\newcommand{\cd}{\mathcal{D}}
\newcommand{\cbr}{\cb_{r}}

\newcommand{\cB}{\mathcal{B}}

\newcommand{\cs}{\mathcal{S}}

\newcommand{\vy}{\mathbf{y}}

\newcommand{\vz}{\mathbf{z}}
\newcommand{\vv}{\mathbf{v}}
\newcommand{\ve}{\mathbf{e}}
\newcommand{\vw}{\mathbf{w}}

\newcommand{\vg}{\mathbf{g}}
\newcommand{\vu}{\mathbf{u}}
\newcommand{\vh}{\mathbf{h}}

\newcommand{\rr}{\mathbb{R}}
\newcommand{\ee}{\mathbb{E}}

\newcommand{\proj}{\Pi}

\graphicspath{{imgs/}}

\theoremstyle{plain} \numberwithin{equation}{section}
\newtheorem{theorem}{Theorem}[section]
\numberwithin{theorem}{section}
\newtheorem{corollary}[theorem]{Corollary}

\newtheorem{lemma}[theorem]{Lemma}
\newtheorem{proposition}[theorem]{Proposition}

\theoremstyle{definition}
\newtheorem{definition}[theorem]{Definition}

\newtheorem{remark}[theorem]{Remark}
\newtheorem{example}[theorem]{Example}

%% file: main.bbl
\begin{thebibliography}{10}

\bibitem{bagirov2014introduction}
A.~Bagirov, N.~Karmitsa, and M.~M. M{\"a}kel{\"a}.
\newblock {\em Introduction to Nonsmooth Optimization: theory, practice and
  software}, volume~12.
\newblock Springer, 2014.

\bibitem{balkanski2019exponential}
E.~Balkanski, A.~Rubinstein, and Y.~Singer.
\newblock An exponential speedup in parallel running time for submodular
  maximization without loss in approximation.
\newblock In {\em Proc. ACM-SIAM SODA}, 2019.

\bibitem{Ball:1997}
K.~Ball.
\newblock An elementary introduction to modern convex geometry.
\newblock {\em Flavors of geometry}, 31:1--58, 1997.

\bibitem{beck2012smoothing}
A.~Beck and M.~Teboulle.
\newblock Smoothing and first order methods: A unified framework.
\newblock {\em SIAM Journal on Optimization}, 22(2):557--580, 2012.

\bibitem{blum2020foundations}
A.~Blum, J.~Hopcroft, and R.~Kannan.
\newblock {\em Foundations of data science}.
\newblock Cambridge University Press, 2020.

\bibitem{bubeck2019complexity}
S.~Bubeck, Q.~Jiang, Y.-T. Lee, Y.~Li, and A.~Sidford.
\newblock Complexity of highly parallel non-smooth convex optimization.
\newblock {\em Advances in neural information processing systems}, 32, 2019.

\bibitem{carmon2017lower-II}
Y.~Carmon, J.~C. Duchi, O.~Hinder, and A.~Sidford.
\newblock Lower bounds for finding stationary points {II}: First-order methods.
\newblock {\em Mathematical Programming}, Sep 2019.

\bibitem{carmon2020acceleration}
Y.~Carmon, A.~Jambulapati, Q.~Jiang, Y.~Jin, Y.~T. Lee, A.~Sidford, and
  K.~Tian.
\newblock Acceleration with a ball optimization oracle.
\newblock {\em Advances in Neural Information Processing Systems},
  33:19052--19063, 2020.

\bibitem{carmon2021thinking}
Y.~Carmon, A.~Jambulapati, Y.~Jin, and A.~Sidford.
\newblock Thinking inside the ball: Near-optimal minimization of the maximal
  loss.
\newblock In {\em Conference on Learning Theory}, 2021.

\bibitem{chakrabarty2024parallel}
D.~Chakrabarty, A.~Graur, H.~Jiang, and A.~Sidford.
\newblock Parallel submodular function minimization.
\newblock {\em Advances in Neural Information Processing Systems}, 36, 2024.

\bibitem{chambolle2011first}
A.~Chambolle and T.~Pock.
\newblock A first-order primal-dual algorithm for convex problems with
  applications to imaging.
\newblock {\em Journal of mathematical imaging and vision}, 40:120--145, 2011.

\bibitem{cohen2018acceleration}
M.~B. Cohen, J.~Diakonikolas, and L.~Orecchia.
\newblock On acceleration with noise-corrupted gradients.
\newblock In {\em International Conference on Machine Learning (ICML)}, 2018.

\bibitem{cui2021modern}
Y.~Cui and J.-S. Pang.
\newblock {\em Modern nonconvex nondifferentiable optimization}.
\newblock SIAM, 2021.

\bibitem{cutkosky2023optimal}
A.~Cutkosky, H.~Mehta, and F.~Orabona.
\newblock Optimal stochastic non-smooth non-convex optimization through
  online-to-non-convex conversion.
\newblock In {\em International Conference on Machine Learning}, 2023.

\bibitem{davis2022gradient}
D.~Davis, D.~Drusvyatskiy, Y.~T. Lee, S.~Padmanabhan, and G.~Ye.
\newblock A gradient sampling method with complexity guarantees for {L}ipschitz
  functions in high and low dimensions.
\newblock {\em Advances in Neural Information Processing Systems},
  35:6692--6703, 2022.

\bibitem{davis2022nearly}
D.~Davis and L.~Jiang.
\newblock A nearly linearly convergent first-order method for nonsmooth
  functions with quadratic growth.
\newblock {\em arXiv preprint arXiv:2205.00064}, 2022.

\bibitem{Devolder:2014}
O.~Devolder, F.~Glineur, and Y.~E. Nesterov.
\newblock First-order methods of smooth convex optimization with inexact
  oracle.
\newblock {\em Math. Program.}, 146(1-2):37--75, 2014.

\bibitem{diakonikolas2020lower}
J.~Diakonikolas and C.~Guzm{\'a}n.
\newblock Lower bounds for parallel and randomized convex optimization.
\newblock {\em The Journal of Machine Learning Research}, 21(1):153--183, 2020.

\bibitem{thegaptechnique}
J.~Diakonikolas and L.~Orecchia.
\newblock The approximate duality gap technique: A unified theory of
  first-order methods.
\newblock {\em SIAM Journal on Optimization}, 29(1):660--689, 2019.

\bibitem{duchi2012randomized}
J.~C. Duchi, P.~L. Bartlett, and M.~J. Wainwright.
\newblock Randomized smoothing for stochastic optimization.
\newblock {\em SIAM Journal on Optimization}, 22(2):674--701, 2012.

\bibitem{flaxman2005online}
A.~D. Flaxman, A.~T. Kalai, and H.~B. McMahan.
\newblock Online convex optimization in the bandit setting: gradient descent
  without a gradient.
\newblock In {\em Proceedings of the sixteenth annual ACM-SIAM Symposium on
  Discrete Algorithms}, pages 385--394, 2005.

\bibitem{gasnikov2018universal}
A.~V. Gasnikov and Y.~E. Nesterov.
\newblock Universal method for stochastic composite optimization problems.
\newblock {\em Computational Mathematics and Mathematical Physics},
  58(1):48--64, 2018.

\bibitem{ghadimi2012optimal}
S.~Ghadimi and G.~Lan.
\newblock Optimal stochastic approximation algorithms for strongly convex
  stochastic composite optimization {I}: A generic algorithmic framework.
\newblock {\em SIAM Journal on Optimization}, 22(4):1469--1492, 2012.

\bibitem{grimmer2018radial}
B.~Grimmer.
\newblock Radial subgradient method.
\newblock {\em SIAM Journal on Optimization}, 28(1):459--469, 2018.

\bibitem{grimmer2019convergence}
B.~Grimmer.
\newblock Convergence rates for deterministic and stochastic subgradient
  methods without lipschitz continuity.
\newblock {\em SIAM Journal on Optimization}, 29(2):1350--1365, 2019.

\bibitem{grimmer2024optimal}
B.~Grimmer.
\newblock On optimal universal first-order methods for minimizing heterogeneous
  sums.
\newblock {\em Optimization Letters}, 18(2):427--445, 2024.

\bibitem{Guler:1991}
O.~G\"{u}ler.
\newblock On the convergence of the proximal point algorithm for convex
  minimization.
\newblock {\em SIAM Journal on Control and Optimization}, 29(2):403--419, 1991.

\bibitem{Guzman:2015}
C.~Guzm\'an and A.~Nemirovski.
\newblock On lower complexity bounds for large-scale smooth convex
  optimization.
\newblock {\em Journal of Complexity}, 31(1):1 -- 14, 2015.

\bibitem{Guler:1992}
O.~Güler.
\newblock New proximal point algorithms for convex minimization.
\newblock {\em SIAM Journal on Optimization}, 2(4):649--664, 1992.

\bibitem{han2023survey}
X.~Han and A.~S. Lewis.
\newblock Survey descent: A multipoint generalization of gradient descent for
  nonsmooth optimization.
\newblock {\em SIAM Journal on Optimization}, 33(1):36--62, 2023.

\bibitem{hazan2015beyond}
E.~Hazan, K.~Levy, and S.~Shalev-Shwartz.
\newblock Beyond convexity: Stochastic quasi-convex optimization.
\newblock {\em Advances in neural information processing systems}, 28, 2015.

\bibitem{hosseini2019nonsmooth}
S.~Hosseini, B.~S. Mordukhovich, and A.~Uschmajew.
\newblock {\em Nonsmooth optimization and its applications}.
\newblock Springer, 2019.

\bibitem{john1961functions}
F.~John and L.~Nirenberg.
\newblock On functions of bounded mean oscillation.
\newblock {\em Communications on pure and applied Mathematics}, 14(3):415--426,
  1961.

\bibitem{jordan2023deterministic}
M.~Jordan, G.~Kornowski, T.~Lin, O.~Shamir, and M.~Zampetakis.
\newblock Deterministic nonsmooth nonconvex optimization.
\newblock In {\em Conference on Learning Theory}, pages 4570--4597, 2023.

\bibitem{Kong:2023}
S.~Kong and A.~Lewis.
\newblock The cost of nonconvexity in deterministic nonsmooth optimization.
\newblock {\em Mathematics of Operations Research}, 2023.

\bibitem{kornowski2022oracle}
G.~Kornowski and O.~Shamir.
\newblock Oracle complexity in nonsmooth nonconvex optimization.
\newblock {\em The Journal of Machine Learning Research}, 23(1):14161--14204,
  2022.

\bibitem{lakshmanan2008decentralized}
H.~Lakshmanan and D.~P. De~Farias.
\newblock Decentralized resource allocation in dynamic networks of agents.
\newblock {\em SIAM Journal on Optimization}, 19(2):911--940, 2008.

\bibitem{lemarechal1978nonsmooth}
C.~Lemarechal.
\newblock Nonsmooth optimization and descent methods.
\newblock 1978.

\bibitem{li2023simple}
T.~Li and G.~Lan.
\newblock A simple uniformly optimal method without line search for convex
  optimization.
\newblock {\em arXiv preprint arXiv:2310.10082}, 2023.

\bibitem{lu2019relative}
H.~Lu.
\newblock “{R}elative continuity” for non-{L}ipschitz nonsmooth convex
  optimization using stochastic (or deterministic) mirror descent.
\newblock {\em INFORMS Journal on Optimization}, 1(4):288--303, 2019.

\bibitem{malitsky2020adaptive}
Y.~Malitsky and K.~Mishchenko.
\newblock Adaptive gradient descent without descent.
\newblock In {\em Proceedings of the 37th International Conference on Machine
  Learning (ICML)(2020)}, volume 119, 2020.

\bibitem{martinet1970regularisation}
B.~Martinet.
\newblock Regularisation, d'in{\'e}quations variationelles par approximations
  succesives.
\newblock {\em Revue Francaise d'informatique et de Recherche operationelle},
  1970.

\bibitem{moreau1965proximite}
J.-J. Moreau.
\newblock Proximit{\'e} et dualit{\'e} dans un espace {H}ilbertien.
\newblock {\em Bulletin de la Soci{\'e}t{\'e} math{\'e}matique de France},
  93:273--299, 1965.

\bibitem{nemirovski1994parallel}
A.~Nemirovski.
\newblock On parallel complexity of nonsmooth convex optimization.
\newblock {\em Journal of Complexity}, 10(4):451--463, 1994.

\bibitem{nemirovski2004prox}
A.~Nemirovski.
\newblock Prox-method with rate of convergence {$O(1/t)$} for variational
  inequalities with {L}ipschitz continuous monotone operators and smooth
  convex-concave saddle point problems.
\newblock {\em SIAM Journal on Optimization}, 15(1):229--251, 2004.

\bibitem{Nemirovski:1985}
A.~Nemirovskii and Y.~Nesterov.
\newblock {Optimal methods of smooth convex optimization {\em (in Russian)}}.
\newblock {\em Zh. Vychisl. Mat. i Mat. Fiz.}, 25(3):356--369, 1985.

\bibitem{nemirovski:1983}
A.~Nemirovskii and Yudin.
\newblock {\em Problem Complexity and Method Efficiency in Optimization}.
\newblock Wiley, 1983.

\bibitem{Nesterov:2005minimizing}
Y.~Nesterov.
\newblock Minimizing functions with bounded variation of subgradients.
\newblock Technical report, CORE Discussion Papers, 2005.

\bibitem{nesterov2005smooth}
Y.~Nesterov.
\newblock Smooth minimization of non-smooth functions.
\newblock {\em Mathematical programming}, 103(1):127--152, 2005.

\bibitem{nesterov2015universal}
Y.~Nesterov.
\newblock Universal gradient methods for convex optimization problems.
\newblock {\em Mathematical Programming}, 152(1-2):381--404, 2015.

\bibitem{nesterov2017random}
Y.~Nesterov and V.~Spokoiny.
\newblock Random gradient-free minimization of convex functions.
\newblock {\em Foundations of Computational Mathematics}, 17:527--566, 2017.

\bibitem{norkin2023constrained}
V.~Norkin, A.~Pichler, and A.~Kozyriev.
\newblock Constrained global optimization by smoothing.
\newblock {\em arXiv preprint arXiv:2308.08422}, 2023.

\bibitem{norkin2020stochastic}
V.~I. Norkin.
\newblock A stochastic smoothing method for nonsmooth global optimization.
\newblock {\em Kibernetika ta komp’iuterni tekhnolohii}, 2020.

\bibitem{pang1997error}
J.-S. Pang.
\newblock Error bounds in mathematical programming.
\newblock {\em Mathematical Programming}, 79(1-3):299--332, 1997.

\bibitem{Rademacher:1919}
H.~Rademacher.
\newblock Über partielle und totale differenzierbarkeit von funktionen
  mehrerer variabeln und über die transformation der doppelintegrale.
\newblock {\em Mathematische Annalen}, 79:340--359, 1919.

\bibitem{renegar2016efficient}
J.~Renegar.
\newblock ``{E}fficient” subgradient methods for general convex optimization.
\newblock {\em SIAM Journal on Optimization}, 26(4):2649--2676, 2016.

\bibitem{rockafellar1976monotone}
R.~T. Rockafellar.
\newblock Monotone operators and the proximal point algorithm.
\newblock {\em SIAM Journal on Control and Optimization}, 14(5):877--898, 1976.

\bibitem{shor2012minimization}
N.~Z. Shor.
\newblock {\em Minimization methods for non-differentiable functions},
  volume~3.
\newblock Springer Science \& Business Media, 2012.

\bibitem{stein1993harmonic}
E.~M. Stein and T.~S. Murphy.
\newblock {\em Harmonic analysis: real-variable methods, orthogonality, and
  oscillatory integrals}, volume~3.
\newblock Princeton University Press, 1993.

\bibitem{stekloff1907expressions}
V.~Steklov.
\newblock Sur les expressions asymptotiques de certaines fonctions,
  d{\'e}finies par les {\'e}quations diff{\'e}rentielles lin{\'e}aires du
  second ordre, et leurs applications au probl{\`e}me du d{\'e}veloppement
  d'une fonction arbitraire en s{\'e}ries proc{\'e}dant suivant les-dites
  fonctions.
\newblock {\em Comm. Charkov Math. Soc.}, 10(2):97--199, 1907.

\bibitem{taylor2017convex}
A.~B. Taylor.
\newblock {\em Convex interpolation and performance estimation of first-order
  methods for convex optimization.}
\newblock PhD thesis, Catholic University of Louvain, Louvain-la-Neuve,
  Belgium, 2017.

\bibitem{Vershynin:2018}
R.~Vershynin.
\newblock {\em High-Dimensional Probability: An Introduction with Applications
  in Data Science}.
\newblock Cambridge Series in Statistical and Probabilistic Mathematics.
  Cambridge University Press, 2018.

\bibitem{wiegerinck1997bmo}
J.~Wiegerinck.
\newblock {BMO}-space.
\newblock {\em Encyclopaedia of Mathematics Supplement Volume I}, pages
  133--134, 1997.

\bibitem{woodworth2018graph}
B.~E. Woodworth, J.~Wang, A.~Smith, B.~McMahan, and N.~Srebro.
\newblock Graph oracle models, lower bounds, and gaps for parallel stochastic
  optimization.
\newblock {\em Advances in neural information processing systems}, 31, 2018.

\bibitem{yousefian2010convex}
F.~Yousefian, A.~Nedi{\'c}, and U.~V. Shanbhag.
\newblock Convex nondifferentiable stochastic optimization: A local randomized
  smoothing technique.
\newblock In {\em Proceedings of the 2010 American Control Conference}, pages
  4875--4880. IEEE, 2010.

\bibitem{zhang2020complexity}
J.~Zhang, H.~Lin, S.~Jegelka, S.~Sra, and A.~Jadbabaie.
\newblock Complexity of finding stationary points of nonconvex nonsmooth
  functions.
\newblock In {\em International Conference on Machine Learning}, pages
  11173--11182, 2020.

\end{thebibliography}
